\documentclass[a4paper,10pt]{article}
\newif\ifuseprecompiled
\useprecompiledtrue 

\usepackage{graphicx}
\usepackage{pgfplots}
\pgfplotsset{compat=newest}

\usepackage{amssymb,amsmath}
\usepackage{color}
\usepackage{mathtools}

\usepackage{amsopn}
\DeclareMathOperator{\diag}{diag}
\DeclareMathOperator{\rank}{rank}
\DeclareMathOperator{\sv}{sv}

\usepackage{tikz}
\usepackage{graphicx}
\usepackage{pgfplots}
\usepackage{import}
\usepackage{tikzscale}

\usetikzlibrary{calc}
\pgfplotsset{compat=newest}

\usepgfplotslibrary{external}
\definecolor{RWTHteal}{RGB}{0,152,161}
\definecolor{RWTHmagenta}{RGB}{138,32,60} 
\definecolor{RWTHsuperlightblue}{RGB}{199,221,242}
\definecolor{RWTHgrey}{RGB}{100,101,103}
\definecolor{RWTHblack}{RGB}{0,0,0}
\definecolor{RWTHbordeaux}{RGB}{161,16,53}
\definecolor{RWTHorange}{RGB}{246,168,0}
\definecolor{RWTHblue}{RGB}{0,83,159}
\definecolor{RWTHwhite}{RGB}{255,255,255}

\newcommand{\greym}[1]{\color{RWTHgrey}{#1}\color{RWTHblack}}

\newlength\figureheight
\newlength\figurewidth

\newcommand{\lhb}[1]{\ensuremath{\mathfrak{L}(#1)}}
\newcommand{\rhb}[1]{\ensuremath{\mathfrak{R}(#1)}}

\newcommand{\R}{\ensuremath{\mathbb{R}}}
\newcommand{\K}{\ensuremath{\mathbb{K}}}
\newcommand{\C}{\ensuremath{\mathbb{C}}}
\newcommand{\N}{\ensuremath{\mathbb{N}}}
\newcommand{\D}{\ensuremath{\mathcal{D}}}
\newcommand{\bigslant}[2]{{\raisebox{.2em}{$#1$}\left/\raisebox{-.2em}{$#2$}\right.}} 

\newcommand{\HONEY}{\ensuremath{\mathrm{HONEY}}}
\newcommand{\BDRY}{\ensuremath{\mathrm{BDRY}}}
\newcommand{\HIVE}{\ensuremath{\mathrm{HIVE}}}
\newcommand{\EDGE}{\ensuremath{\mathrm{edge}}}

\newcommand{\ck}{\boxtimes}

\DeclareMathOperator{\degree}{deg}

\usepackage{etoolbox}
\makeatletter
\patchcmd{\@addmarginpar}{\ifodd\c@page}{\ifodd\c@page\@tempcnta\m@ne}{}{}
\makeatother
\reversemarginpar

\def\BC#1#2{}

\def\new#1{#1}

\def\old#1{}
\def\repl#1#2{}

\def\vec{\mathrm{vec}}

\usepackage{subfig}
\usepackage{amsmath,amssymb,amsthm}
\usepackage{amsfonts}
\usepackage{color}
\usepackage{algorithm}
\usepackage{algorithmic}
\usepackage{multirow}
\usepackage{multicol}
\usepackage{relsize,amsmath} 
\usepackage{enumitem} 
\setlist[enumerate]{itemsep=0pt,topsep=0pt,parsep=1pt} 
\setlist[itemize]{itemsep=0pt,topsep=0pt,parsep=1pt} 

\usepackage{placeins} 
\usepackage{mathtools} 
\usepackage{titlesec} 
\titleformat*{\section}{\large \bfseries}
\titleformat*{\subsection}{\normalsize \bfseries}
\titleformat*{\subsubsection}{\normalsize \bfseries}

\usepackage[font={footnotesize,it}]{caption} 
\usepackage{cite} 
\usepackage[colorlinks   = true,linkcolor=RWTHblue,citecolor=RWTHblue]{hyperref}
\numberwithin{equation}{section} 

\usepackage{thmtools} 
\declaretheoremstyle[
    headfont=\bfseries,
    notefont=\normalfont,
    bodyfont=\itshape,
    postheadhook={},
]{mytheorem}

\declaretheorem[style=mytheorem,numberwithin=section]{theorem}
\declaretheorem[style=mytheorem,numberlike=theorem]{corollary}

\declaretheorem[style=mytheorem,numberlike=theorem]{lemma}
\declaretheorem[style=mytheorem,numberlike=theorem]{proposition}

\declaretheorem[style=mytheorem,numberlike=theorem]{definition}
\declaretheorem[style=mytheorem,numberlike=theorem]{remark}

\evensidemargin=\oddsidemargin
\newdimen\dummy
\dummy=\oddsidemargin
\usepackage[capitalise]{cleveref}

\author{Sebastian Kr\"amer\thanks{Institut f\"ur Geometrie und Praktische Mathematik,  RWTH Aachen University,  Templergraben 55,
52056 Aachen, Germany
  ({\tt kraemer@igpm.rwth-aachen.de}, {\tt http://www.igpm.rwth-aachen.de}).}
}

\title{A Geometric Description of Feasible Singular Values in the Tensor Train Format} 
\date{}
\begin{document}

\maketitle

\begin{abstract}
Tree tensor networks\BC{b1_}{b_} such as the tensor train format are a common tool for high dimensional problems.
The associated multivariate rank and accordant tuples of singular values are based on different matricizations of the same tensor.
While the behavior of such is as essential as in the matrix case, 
here the question about the \textit{feasibility} of specific constellations
arises: which prescribed tuples can be realized as singular values of a tensor and
what is this feasible set? \\
We first show the equivalence of the \textit{tensor feasibility problem (TFP)} to the \textit{quantum marginal problem (QMP)}.
In higher dimensions, in case of the tensor train (TT-)format, the conditions for feasibility can be decoupled.
By present results for three dimensions for the QMP, it then follows that the tuples of squared, feasible TT-singular values form polyhedral cones.
We further establish a connection to eigenvalue relations of sums of Hermitian matrices, 
which in turn are described by sets of interlinked, so called \textit{honeycombs}, as they have been introduced by Knutson and Tao.\\
Besides a large class of universal, necessary inequalities as well as the vertex description for a special, simpler instance, we present
a linear programming algorithm to check feasibility and a simple, heuristic algorithm to construct representations of tensors with prescribed, feasible TT-singular values
in parallel.\\


\end{abstract} 
\section{Introduction}\label{sec:intro}
For $\K \in \{\R,\C\}$\BC{q1}{q}, let $A\in\K^{n_1 \times \ldots \times n_d}$ be a $d$th-order tensor, such as in \cref{fig:tensor4d}.
\begin{figure}[tbhp]
  \centering
  \ifuseprecompiled
  \includegraphics{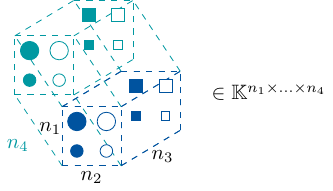}
  \else
  \setlength\figureheight{3cm}
  \setlength\figurewidth{12cm}
  \tikzsetnextfilename{tensor4d}
  \input{tikz_base_files/reshaping_neu/tensor4d.tikz}
  \fi
  \caption{Visualization of a $4-$dimensional tensor $A$ with mode sizes $n_1 = \ldots = n_4 = 2$.}
  \label{fig:tensor4d} 
\end{figure}%
The tensor $A$ allows to be reshaped into certain matricizations
\begin{align*}
  A^{(\{1,\ldots,\mu\})} \in \K^{n_1\cdot\ldots \cdot n_\mu\ \times \ n_{\mu+1}\cdot \ldots \cdot n_d}, \quad \mu = 1,\ldots,d-1,
\end{align*}
which are related to the so called tensor train (TT-)decomposition \cite{Gr10_Hie,Os11_Ten}.
The vectorization \new{$\vec(\cdot)$}\footnote{\new{in Matlab syntax, $\vec(A) = A(:)$}} in co-lexicographic order (column-wise) is to be an invariant to these reshapings\BC{a1}{a}, \new{i.e.
\begin{align*}
  \vec(A^{(\{1,\ldots,\mu\})}) = \vec(A) \in \K^{n_1\cdot\ldots\cdot n_d \times 1}, \quad \mu = 1,\ldots,d-1,
\end{align*}}
such that they become uniquely defined.\\
We may also explicitly write 
  $A^{(\{1,\ldots,\mu\})}((i_1,\ldots,i_\mu)_1,\, (i_{\mu+1},\ldots,i_d)_{\mu+1}) = A(i)$ 
where $(i_1,\ldots,i_\mu)_\nu := 1+\sum_{s = 1}^\mu \left( \prod_{h = 1}^{s-1} n_{\nu+h-1} \right) (i_s - 1) \in \{1,\ldots,n_\nu \ldots n_{\nu+\mu-1}\} \subset \mathbb{N}$ (we will
skip the index $\nu$ when context renders it obsolete).\\
The $d-1$ tuples of TT-singular values $\sigma = (\sigma^{(1)},\ldots,\sigma^{(d-1)}) = \sv_{\mathrm{TT}}(A)$
and \new{the }according TT-rank\new{(}s\new{)} \old{$r_1,\ldots,r_{d-1} \in \mathbb{N}$ }\BC{d1}{d}\new{$r = (r_1,\ldots,r_{d-1}) \in \N^{d-1}$ } of $A \neq 0$
are given through the matrix singular values $(\sv)$ of its reshapings
\begin{align*}
  \sigma^{(\mu)} := \sv(A^{(\{1,\ldots,\mu\})}), \ r_\mu := \rank(A^{(\{1,\ldots,\mu\})}), \quad \mu = 1,\ldots,d-1,
\end{align*}
such as displayed in \cref{fig:reshapingfig}. 
%
%
%
\begin{figure}[tbhp]
  \centering
  \ifuseprecompiled
  \includegraphics{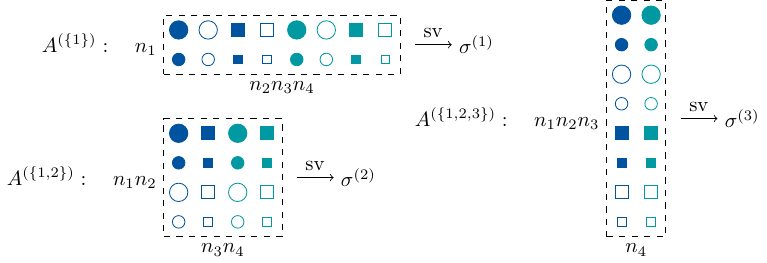}
  \else
  \setlength\figureheight{3cm}
  \setlength\figurewidth{12cm}
  \tikzsetnextfilename{reshaping}
  \input{tikz_base_files/reshaping_neu/reshaping.tikz}
  \fi
  \caption{\label{fig:reshapingfig} Matricizations (or reshapings, or unfoldings) of the $4$th-order tensor $A$ into the matrices $A^{(\{1\})} \in \K^{n_1 \times n_2 n_3 n_4}$, $A^{(\{1,2\})} \in \K^{n_1n_2 \times n_3 n_4}$ and 
    $A^{(\{1,2,3\})} \in \K^{n_1 n_2 n_3 \times n_4}$ through which $\sigma = \sv_\mathrm{TT}(A)$ is obtained.}
\end{figure}%
\noindent
%
%
%
For simplicity, each of these singular values is considered to be a weakly decreasing, infinite sequence with finitely many nonzero entries --- hence element of the \old{space }\new{cone }which we denote with 
$\mathcal{D}^{\infty}_{\geq 0}$ (cf. \cref{def:weakdectuple}).
Since each $\sigma^{(\mu)} = (\sigma^{(\mu)}_1,\sigma^{(\mu)}_2,\ldots)$ is based on the same entries within $A$, the question about the \textit{feasibility}
of prescribed  TT-singular value arises immediately:

\begin{definition}[TT-feasibility]\label{def:deffeasibility}
  Let $\sigma=(\sigma^{(1)},\ldots,\sigma^{(d-1)}) \in (\mathcal{D}^{\infty}_{\geq 0})^{d-1}$. Then $\sigma$ is called {\normalfont feasible
    for $n = (n_1,\ldots,n_d)$} if there exists a tensor $A \in \K^{n_1 \times \ldots \times n_d}$ giving rise to it in form of its TT-singular values,
  i.e. $\sigma = \sv_\mathrm{TT}(A)$.
\end{definition}
In other words, we ask which $\sigma$ are in the range of $\sv_\mathrm{TT}$. One necessary condition, 
\begin{align} \|A\|_F = \|\sigma^{(\mu)}\|_2 = \|\sigma^{(\nu)}\|_2, \quad \mu,\nu = 1,\ldots,d-1, \label{eq:traceprop} \end{align}
where $\|\cdot\|_F$ is the Frobenius norm, follows directly and is denoted as \textit{trace property}. 
\\\\
Understanding the nature of this and related problems regarding low rank decompositions is essential, given
that many methods rely on at least basic, if not rigorous assumptions about these generalized singular values, 
just as in the matrix case. At the same time, they are a key tool to the (complexity) analysis of higher order data,
whether for example in signal processing, machine learning or quantum chemistry,
as presented in the extensive survey articles \cite{CiMaLaZhZhCaPh15_Ten,GrKrTo13_Ali,KoBa09_Ten,SiLaFuHuPaFa17_Ten}.\BC{c1}{c}\\\\
\new{As we will see in \cref{sec:weylhorn}, $\sigma$ is feasible for $n$ in the case $\K = \C$ if and only if it is already feasible for $\K = \R$.
The choice of $\K$ does hence not influence the range of $\mathrm{sv}_{TT}$.\BC{z7}{z}}

\subsection{\new{Equivalence of the tensor feasibility and the quantum marginal problem}}\label{sec:QMPeqTFP}\BC{x1}{x}
Let $I = \{1,\ldots,d-1\}$ and $\mathcal{K} \subset \{ J \mid \emptyset \neq J \subset I\}$ be a family of subsets.
The quantum marginal problem (see for example \cite{DaHa05_Qua,Kl06_Qua,Sc14_Qua}) and the tensor feasibility problem, as defined below, are equivalent in the sense of \cref{def:eqTFPandQMP}.\\\\
The mapping $\mathrm{trace}_{I\setminus J}$, called \textit{partial trace},
is induced via $\mathrm{trace}_{I\setminus J}(A_1 \otimes \ldots \otimes A_{d-1}) = \prod_{i \notin J} \mathrm{trace}(A_i) \cdot \bigotimes_{i \in J} A_i \in \K^{n_J \times n_J}$
for matrices $A_i \in \K^{n_i \times n_i}$ and $n_J := \prod_{i \in J} n_i$ for $J \subset I$.
\begin{definition}[Quantum marginal problem (QMP)]
  For each $J \in \mathcal{K}$, $J \subset I$, let $\lambda^{(J)} \in \mathcal{D}^{\infty}_{\geq 0}$ (potential eigenvalues).
 Then the collection $\{\lambda^{(J)}\}_{J \in \mathcal{K}}$ is called {\normalfont compatible (for $n$)} if there exists a hermitian, positive semi-definite matrix
 $\rho_I \in \C^{n_I \times n_I}$ such that
 \begin{align}
  \mathrm{eig}(\rho_J) = \lambda^{(J)}, \quad \rho_J = \mathrm{trace}_{I \setminus J}(\rho_I) \in \C^{n_J \times n_J}, \label{defcomp}
 \end{align}
 for all $J \in \mathcal{K}$. 
\end{definition}
\begin{definition}[Tensor feasibility problem (TFP) for \texorpdfstring{$\K = \C$}{K=C}]
 For each $J \in \mathcal{K}$, $J \subset I$, let $\sigma^{(J)} \in \mathcal{D}^{\infty}_{\geq 0}$ (potential singular values). 
 Then the collection $\{\sigma^{(J)}\}_{J \in \mathcal{K}}$ is called {\normalfont feasible (for $n$)} if there exists a tensor $A \in \C^{n_1 \times \ldots \times n_d}$
 such that 
 \begin{align}
  \mathrm{sv}(A^{(J)}) = \sigma^{(J)}, \quad A^{(J)} \in \C^{n_J \times n_{I \setminus J}}, \label{deffeas}
 \end{align}
for all $J \in \mathcal{K}$. The matrices $A^{(J)}$ are analogous reshapings and formally defined in, for example, \cite{Gr10_Hie}.
\end{definition}
Sets for which $d \in J$ need not be included in the definition of feasibility, since simply $\mathrm{sv}(A^{(J)}) = \mathrm{sv}(A^{(\{1,\ldots,d\} \setminus J)})$,
$\{1,\ldots,d\} \setminus J \subset I$. Note that in the introduction, and all subsequent sections, we use the short notation $\sigma^{(\mu)} = \sigma^{(\{1,\ldots,\mu\})}$ for the TT-format.
\begin{theorem}[Equivalence of TFP and QMP]\label{def:eqTFPandQMP}
The \textit{feasibility} of $\{\sigma^{(J)}\}_{J \in \mathcal{K}}$ is equivalent to the 
\textit{compatibility} of the entry-wise squared values $\{(\sigma^{(J)})^2\}_{J \in \mathcal{K}}$, where $n_d$ may be chosen as large as necessary
(although at most $n_d = \mathrm{rank}(\rho_I)$ is required).
\end{theorem}
\begin{proof} (of \cref{def:eqTFPandQMP}) \label{proofofTFPQMP}
Equivalence is achieved by setting
\begin{align}\label{eq:relation}
 A^{(I)} {A^{(I)}}^H = \rho_I,
\end{align}
 where $\cdot^H$ is the conjugate (also called Hermitian) transpose. The rest follows by
the simple fact that $\mathrm{trace}_{I \setminus J}(A^{(I)} {A^{(I)}}^H) = A^{(J)} {A^{(J)}}^H$ and hence
\begin{align}\label{eq:chaineq}
\lambda^{(J)} = \mathrm{eig}(\rho_J) = \mathrm{eig}(\mathrm{trace}_{I \setminus J}(A^{(I)} {A^{(I)}}^H)) = \mathrm{sv}(A^{(J)})^2 = (\sigma^{(J)})^2 
\end{align}
for all $J \subset \{1,\ldots,d\}$: First, let $\{\sigma^{(J)}\}_{J \in \mathcal{K}}$ be feasible for $n \in \N^d$ by means of the tensor $A$ as in \eqref{deffeas}. 
Then by \cref{eq:relation,eq:chaineq} the family $\{(\sigma^{(J)})^2\}_{J \in \mathcal{K}}$ is compatible. Conversely, assume the family is compatible by means of $\rho_I$ as in \cref{defcomp}.
 Then we define the tensor $A \in \C^{n_1 \times \ldots \times n_{d-1} \times \mathrm{rank}(\rho_I)}$
 via the Cholesky decomposition of $\rho_I$ as in \cref{eq:relation}. Hence, via \cref{eq:chaineq}, the family $\{\sigma^{(J)}\}_{J \in \mathcal{K}}$ is feasible
 for any $n_d \geq \mathrm{rank}(\rho_I)$.
\end{proof}
The \textit{pure} QMP adds the condition $\mathrm{rank}(\rho_I) = 1$. To obtain the equivalent TFP,
one sets $n_d = 1$. For two dimensions, the problem is reduced to the ordinary matrix singular values by which
$\sigma^{(\{1\})} = \sigma^{(\{2\})}$. 
For three dimensions, the relation between the pure QMP and the TFP for $\mathcal{K} = \{\{1\},\{2\},\{3\}\}$ is commonly mentioned, e.g. in \cite{Kl06_Qua}.
More general, the pure QMP for $\mathcal{K} = \{\{1\},\{2\},\ldots\}$ which is concerned with the spectra of $\rho_{\{1\}},\rho_{\{2\}},\ldots$ (often denoted
as density matrices $\rho_A,\rho_B,\ldots$) is 
the same as the Tucker-feasibility problem (cf. \cite{DoStLa17_Ont}).
Since the tensor space is effectively reduced by one dimension, one may substitute $d \leftarrow d-1$ and use $\widetilde{\mathcal{K}} = \{ \{1\},\ldots,\{d-2\},\{1,\ldots,d-1\} \}$,
which reveals that the pure QMP for $\mathcal{K}$ is equivalent to the QMP for $\widetilde{\mathcal{K}}$. For dimension $3$,
this equivalence is stated in \cite{Kl06_Qua} (using the notation $\rho_A,\rho_B,\rho_{AB}$ and $\rho_A,\rho_B,\rho_C$).
\\\\
The TT-feasibility problem in turn is identified with the quantum marginal problem for $\mathcal{K} = \{\{1\},\{1,2\},\ldots,\{1,\ldots,d-1\}\}$,
that is, the problem which is concerned with the spectra of $\rho_{\{1\}},\rho_{\{1,2\}},\ldots,\rho_{\{1,\ldots,d-1\}}$ (often denoted as density matrices $\rho_A,\rho_{AB},\ldots$).
The feasibility problem may demand an additional constraint $n_d < \infty$ which however only restricts $\degree(\sigma^{(I)}) = \mathrm{rank}(\rho_I) \leq n_d$.

\subsection{\new{The quantum marginal problem}}\label{sec:QMC}
\new{Earlier articles have answered several special instances of the QMP, which suggest that sets of compatible values
form convex, closed cones:
\\\\
\textit{Pure QMP for $\mathcal{K} = \{\{1\},\ldots,\{d-1\}\}$ (Tucker-feasibility)}:
For $n_i = 2$, $i = 1,\dots,d$, the physical interpretation of the pure QMP is related to an array of qubits. 
For every $d \in \N$, it is governed by the simple inequalities 
\begin{align}
 \lambda_2^{(\{i\})} \leq \sum_{j (\neq i)} \lambda_2^{(\{j\})}, \quad i \in I,
\end{align}
as proven by \cite{HiSuSz03_One} (cf. \cref{sec:feashist}). All constraints for the pure QMP with $n_i = 3$, $i = 1,\ldots,d$ for $d-1 = 3$, have been derived in \cite{Fr02_Mom,Hi03_Ont}.
Subsequently, \cite{Kl06_Qua} has presented a general solution to the pure QMP for $\mathcal{K} = \{ \{1\},\{2\},\{3\} \}$ for arbitrary $n$,
based on geometric invariant theory, and states that the cases $d - 1 > 3$ are straightforward.\\\\
\textit{QMP for $\mathcal{K} = \{\{1\},\{1,2\}\}$ (TT-feasibility for $d-1 = 2$)}: 
Similarly, \cite{DaHa05_Qua} has provided an elaborate answer to this QMP in form of a relation between cohomologies of Grassmannians. 
For each specific $n_1$ and $n_2$, a finite set of linear inequalities can thereby be derived which are equivalent (cf. \cite{BeSj00_Coa}) to compatibility.
Although the two latter solutions are in a certain sense complete (from an algebraic perspective), \cite{DaHa05_Qua} could for example only conjecture that in the
special case $n_1 \leq n_2$, compatibility of $(\lambda^{\{1\}}, \lambda^{\{1,2\}})$ is equivalent to just 
\begin{align}\label{eq:n2leqn1}
 \sum_{i=1}^k \lambda^{(\{1\})}_i \leq \sum_{i=1}^{n_2 k} \lambda^{(\{1,2\})}_i, \quad k = 1,\ldots,n_1,
\end{align}
where equality must hold for $k = n_1$ (which relates to the trace property for feasibility).
This instance was later confirmed by \cite{LiPoWa14_Ran} (in again different notation).\\\\
\textit{QMP for hierarchically structured $\mathcal{K}$}:
Interestingly, other classes of families $\mathcal{K}$ pose open problems, but may be approached through tensor format theory, such as for the TT-format. If the sets in $\mathcal{K}$
fulfill the hierarchy condition (\cite{Gr10_Hie})
\begin{align}
  J \cap \widetilde{J} \in \{ \emptyset,\ J,\ \widetilde{J} \}, \quad \mbox{for all $J,\widetilde{J} \in \mathcal{K}$},
\end{align}
then the equivalent feasibility problem can be decoupled into three-dimensional sub-problems using a hierarchical standard representation (for both $\K = \C$ and $\K = \R$) analogous to the one
we define in \cref{prop:stre}. We will however restrict ourselves to the TT-format here, since the general tensor tree network case is
beyond the scope of this paper. Also, the TT-format poses a certain special case as it corresponds to a tree graph in which each node is connected with at most two other ones.
}

\subsection{Overview \old{over}\new{of }results in this work}\label{sec:results}\BC{d1_}{d_}
\new{We show that if $\sigma$ and $\tau$ are feasible for $n \in \N^d$ (in the sense of \cref{def:deffeasibility}), then $\upsilon := \sqrt{\sigma^2+\tau^2}$, evaluated entry-wise, is feasible for $n$ as well (\cref{cor:Pyt}). 
This means that the set of squared feasible TT-singular values forms a
convex cone, which is closed and finitely generated, as it is to be expected\BC{d2_}{d_} from earlier QMP results on other families of matricizations. \BC{z0}{z}
This result is based on a decoupling (\cref{prop:stre}), by which we prove that
the single conditions for neighboring pairs $(\gamma,\theta) = (\sigma^{(\mu-1)},\sigma^{(\mu)})$ of singular values
already provide all conditions for the higher dimensional case (\cref{cor:pairfeas}). Further, these conditions are independent of $\K \in \{\R,\C\}$ (cf. \cref{thm:fultonsym}).}
\old{
The most characteristic result \repl{(Corollary 2.5)}{(\cref{cor:Pyt})} which we show, and quite the contrary a nontrivial result, implies
the following: If $\sigma$ and $\tau$ are feasible for $n \in \R^d$, then $\upsilon := \sqrt{\sigma^2+\tau^2}$ (entry-wise
evaluation) is feasible for $n$ as well. 
This means that the set of squared feasible TT-singular values forms a
convex cone, and as it turns out it is closed and finitely generated. 
This result is based on a decoupling \repl{(Definition 2.6 and Theorem 2.7)}{(\cref{thm:fundthm})}, by which it suffices to regard neighboring pairs $(\gamma,\theta) = (\sigma^{(\mu)},\sigma^{(\mu+1)})$ of singular values.}\\
\new{
Our slightly different perspective on feasibility of pairs (that is $d=3$) leads to the investigation of sets of interconnected (\cref{{def:hive}}), so called \textit{honeycombs} \cite{KnTa01_Hon}.
Apart from a pleasant graphical depiction (e.g. \cref{fig:plot_triv_feasible_single}), these constructs are at the same time a universal linear programming tool (\cref{alg:linprog})
which can decide the feasibility of each single pair $(\gamma,\theta)$ with low order polynomial computational complexity.
Hence, we can thereby also decide the feasibility of TT-singular values $\sigma$.
We further provide classes of necessary, linear inequalities (\cref{prop:setofineq}) for arbitrary $n \in \N$
and revisit the above mentioned special case \cref{eq:n2leqn1}, providing a complete vertex description as well (\cref{completevert})\BC{ka1}{ka}.
}%
\old{%
We further provide significant, necessary linear inequalities \repl{(Proposition 4.7, Theorem 5.1)}{(\cref{prop:setofineq})}, possibly describing facets.
If for each $\mu = 1,\ldots,d-1$ it holds $r_{\mu-1},r_{\mu} \leq n_{\mu}$, then feasibility is equivalent
to the validity only of the trace property \repl{(1.1) (cf. Theorem 2.10)}{\cref{eq:traceprop} (cf. \cref{thm:prop1})}. This is a generalization 
of the case $d=2$, where the only restriction is that
the number of positive singular values does not exceed the size of the matrix. }%
\new{Last but not least, we provide algorithms to construct tensors with prescribed, feasible
singular values in parallel (\cref{sec:diagonalfeas,alg:numfea}).}
\subsection{\new{Other results on the feasibility problem}\old{Other sets of matricizations}}\label{sec:feashist}
Although many results can be overtaken from the QMP (see \cref{sec:QMC}), we will here give a history of the so far independently approached feasibility problem.
For higher--order tensors, several notions of ranks exist, of which the tensor train
and Tucker format (or HOSVD) \cite{LaMoVa00_AMu,Tucker66_Som} capture two particular ones. The problem of feasibility has originally been introduced and defined
by \cite{HaUs17_Ont} for the Tucker decomposition. Shortly afterwards, further steps have been taken in \cite{HaKrUs17_Per}, from which we have overtaken several notations.
However\BC{e1_}{e_} due to the difference between the two mentioned formats, no results could, so far, be transferred.
Through matrix analysis and eigenvalue relations, \cite{DoStLa17_Ont} later introduced
necessary and sufficient linear inequalities
regarding feasibility mostly restricted to the largest Tucker-singular values of tensors with one common mode size. 
Independently, \cite{Seigal18_Gra} proved the same result for the Tucker format provided $n_1 = \ldots = n_d = 2$ using yet other approaches within algebraic geometry. 
\\
This article, on the other hand, is based on a reduction through gauge conditions to coupled, pairwise problems which are then linked to eigenvalue problems and so called honeycombs \cite{KnTa01_Hon}.
\new{In our tensor train case, which to the best of our knowledge has not been dealt with before, honeycombs
as well as \cite{DaHa05_Qua} fortunately provide both a theoretical and practical resolution to the simpler pairwise problem (see \cref{sec:results}). }%
\old{In our tensor train case, which to the best of our knowledge has not been dealt with before, honeycombs and the Horn conjecture fortunately provide a near full resolution to the problem (see \repl{subsection 1.3}{\cref{sec:results})}.}%
The connection of feasibility to the Horn conjecture has, to a smaller extent, also
synchronously and again independently been investigated by the afore mentioned article \cite{DoStLa17_Ont}, as they deal with yet different eigenvalue problems.
\old{We like to remark that the initial approach from our work is in fact transferable.}\new{As already mentioned in \cref{sec:QMC}, } an analogous way of decoupling can
be applied to the Tucker case and indeed any other hierarchical format,
so that any such feasibility problem for a $d$th-order tensor can be reduced to the pairwise problems as in the tensor train format 
and/or the Tucker format in three dimensions.

%
%
%
%
\subsection{Organization of article}\label{sec:orga}
\new{\BC{kb1}{kb}%
In \cref{sec:reduction}, we use the \textit{standard representation},
an essentially unique representation which meets important gauge conditions,
to reduce the problem of TT-feasibility to only pairs of tuples of singular values.\
In \cref{sec:feasofpairs}, we show the relation to the Horn conjecture, 
give a short overview of related results, and
apply these to our problem with the help of \textit{honeycombs} in \cref{sec:honeycomb}.
We thereby identify the topological structure of sets
of squared TT-feasible singular values as cones, which we further investigate in \cref{sec:coneofsfv}.
Related algorithms can be found in \cref{sec:algo}.%
}


\section{Reduction to mode-wise eigenvalues problems}\label{sec:reduction}
For simplicity, for the remainder of the article, we set $r_0 = r_d = 1$ as well as $\sigma_+^{(0)} = \sigma_+^{(d)} = 1$.
The set of all tensors with (TT-)rank $r$ is denoted as $TT(r) \subset \K^{n_1 \times \ldots \times n_d}$ (\cite{Os11_Ten}). This set
is closely related to so called \textit{representations} (or decompositions) \new{$G = (G_1,\ldots,G_d)$, 
where each so called \textit{core} $G_\mu \in (\K^{r_{\mu-1}\times r_\mu})^{\{1,\ldots,n_\mu\}}$ is an array of matrices $G_\mu(i_\mu)\in \K^{r_{\mu-1}\times r_\mu}$, $i_{\mu} = 1,\ldots,n_\mu$, for $\mu = 1,\ldots,d$.
The product $\boxtimes$ which we define for such in \cref{def:TTformat} can be viewed as generalization of the outer product $\otimes$ for vectors in $\K^{n_\mu} \cong (\K^{1 \times 1})^{\{1,\ldots,n_\mu\}}$.
For now we call $r = (r_1,\ldots,r_{d-1}) \in \N^{d-1}$\BC{d2}{d} the size of $G$ (cf. \cref{thm:TTex}). 
} \old{. For simplicity, we set $r_0 = r_d = 1$.}

\new{%
\begin{definition}[Representation map]\label{def:TTformat}\BC{e1}{e}
For representations $G$ of size $r \in \N^{d-1}$ as above, we define the {\normalfont representation map} $\tau_r$ via 
\begin{align*}
 \tau_r: & \bigtimes_{\mu = 1}^d (\K^{r_{\mu-1}\times r_\mu})^{\{1,\ldots,n_\mu\}} \rightarrow \K^{n_1 \times \ldots \times n_d}, \quad \tau_r(G) := A
\end{align*}
where each entry of the tensor $A$ is a product of matrices in $G$,
\begin{align*}
 A(i_1,\ldots,i_d) := G_1(i_1)\cdot\ldots\cdot G_d(i_d),\quad \forall i \in \bigtimes_{\mu = 1}^d \{1,\ldots,n_\mu\}.
\end{align*}
We further define the associative product $\ck$ for cores $H, N$ via the matrix products $(H \ck N)(i,j) := H(i) \cdot N(j)$,
which generalizes to $A = G_1 \ck \ldots \ck G_d$.
We may skip the symbol $\ck$ in products of a core and matrix (interpreting matrices as cores of length one).
\end{definition}
The cores $G_\mu$ are often also treated as three dimensional tensors, whereas the emphasizing notation we use here
stems from the \textit{matrix product states (MPS) format} \cite{Vi03_Eff}.
 The TT-SVD\footnote{although called SVD, the singular values do not explicitly appear in the decomposition
 as in the matrix SVD}, a generalization of the matrix SVD, provides the following theorem:
 \begin{theorem}[\cite{Os11_Ten}]\label{thm:TTex}\BC{e2}{e}\BC{f1}{f}
  It holds $\mathrm{range}(\tau_r) = \bigcup_{ \widetilde{r} \leq r } TT(\widetilde{r})$, 
where $\widetilde{r} \leq r \in \N^{d-1}$ is to be read entry-wise.
 \end{theorem}
Hence, for every tensor with (TT-)rank $r$, $A \in TT(r)$, there is a representation $G$ of size $r$ for which $A = \tau_r(G)$. 
One therefore also says $G$ has rank $r$.}
These representations will allow us to change the perspective on feasibility and reduce
the problem from a $(d-1)$-tuple to local, pairwise problems.
\begin{definition}[Left and right unfoldings]\label{def:unfoldings}
  For a core $H$ with $H(i) \in \K^{k_1 \times k_2}$, $i = 1,\ldots,\new{m}$,
  the {\normalfont left unfolding} $\lhb{H}\in\K^{k_1 \cdot \new{m} \times k_2}$ is obtained by stacking the matrices $H(i)$ on top of each other in one column and
  likewise the {\normalfont right unfolding} $\rhb{H}\in\K^{k_1 \times k_2 \cdot \new{m}}$ is formed by stacking the same matrices, but side by side, in one row.
  In explicit,
  \begin{align*}
    \left(\lhb{H}\right)_{(\ell,j), q}:= \left(H(j)\right)_{\ell,q}, \quad \left(\rhb{H}\right)_{\ell,(q,j)}
    &:= \left(H(j)\right)_{\ell,q},\quad 
  \end{align*}
  for $1\le j\le \new{m}$, $1\le\ell\le k_1$ and $1\le q\le k_2$. \\
  $H$ is called {\normalfont left-unitary}\BC{k1}{k}
  if $\lhb{H}$ is column-unitary, and {\normalfont right-unitary} if $\rhb{H}$ is row-unitary\footnote{for $\K = \R$, unitary is just orthonormal}.
  For a representation $G$, we
  correspondingly define the {\normalfont interface matrices}\BC{g2}{g}
  \begin{align*}
    G^{\leq \mu} & = \lhb{G_1 \ck \ldots \ck G_{\mu}} \in \K^{n_1\ldots n_{\mu} \times r_{\mu}},\\
    G^{\geq \mu} & = \rhb{G_{\mu} \ck \ldots \ck G_d} \in \K^{r_{\mu-1} \times n_{\mu}\ldots n_d}, \quad \new{\mu = 1,\ldots,d}.
  \end{align*} 
  We also use $G^{<\mu} := G^{\leq \mu-1}$ and $G^{>\mu} := G^{\geq \mu+1}$.
\end{definition}
\new{For any tensor $A  = \tau_r(G)$ it hence holds \BC{g1}{g}\BC{h1}{h}
\begin{align*}
 A^{(\{1,\ldots,\mu\})} = G^{\leq \mu} \ G^{> \mu}, \quad \mu = 1,\ldots,d.
\end{align*}}
The map $\tau_r$ is not injective.
However, there is an essentially unique \textit{standard representation} (in terms of uniqueness of the truncated matrix SVD\footnote{Both $U \Sigma V^H$ and $\widetilde{U} \Sigma \widetilde{V}^H$ are truncated SVDs of $A$ if and only if there exists
an unitary matrix $w$ that commutes with $\Sigma$ and for which $\widetilde{U} = U w$ and $\widetilde{V} = V w$.
For any subset of pairwise distinct nonzero singular values, the corresponding submatrix of $w$ needs to be diagonal with entries in $\{ z \in \K \mid |z| = 1 \}$.}).
In the context of matrix product states, it has priorly appeared in \cite{Vi03_Eff} and is frequently referred to as \textit{canonical MPS}.
\new{Instead of just $d$ cores, this extended representation also contains the tuple of TT-singular values. For that matter, it easy to verify that if both $H$ and $N$ are left- or right-unitary, then $H \boxtimes N$ is
left- or right-unitary, respectively.}

\begin{proposition}[Standard representation]\label{prop:stre}\BC{i1}{i}
 Let $A \in \K^{n_1 \times \ldots \times n_d}$ be a tensor and $\Sigma^{(1)} = \diag(\sigma_+^{(1)}),\ \ldots,\ \Sigma^{(d-1)} = \diag(\sigma_+^{(d-1)})$ 
 be square diagonal matrices which contain the positive TT-singular values of $A$. 
 \new{Then there exists an essentially 
 unique (extended) representation
 \begin{align*} \mathcal{G}^\sigma := (\mathcal{G},\sigma) := (\mathcal{G}_1,\sigma^{(1)},\mathcal{G}_2,\sigma^{(2)},\ldots,\mathcal{G}_{d-1},\sigma^{(d-1)},\mathcal{G}_d), \end{align*}
 with cores $\mathcal{G}_\mu \in (\K^{r_{\mu-1} \times r_{\mu}})^{\{1,\ldots,n_\mu\}}$, $r_\mu = \degree(\sigma^{(\mu)})$, $\mu = 1,\ldots,d$,
 for which the following property holds: 
 \begin{enumerate}
 \item[(1)] For each ${\mu} = 1,\ldots,d-1$,
 \begin{align} 
 \label{eq:sSVD} \lhb{\mathcal{G}_1 \ck \Sigma^{(1)} \mathcal{G}_2 \ck \ldots \ck \Sigma^{({\mu}-1)} \mathcal{G}_{{\mu}}} \quad \Sigma^{(\mu)} \quad \rhb{\mathcal{G}_{{\mu+1}} \Sigma^{({\mu}+1)} \ck \ldots \ck \mathcal{G}_d}
 \end{align}
 is a (truncated) matrix SVD of $A^{(\{1,\ldots,{\mu}\})}$.
\end{enumerate}
 Essentially unique here means that for any other such representation $\widetilde{\mathcal{G}}^\sigma$, it holds $\widetilde{\mathcal{G}}_\mu = w_{\mu-1}^H \mathcal{G}_\mu w_{\mu}$, $\mu = 1,\ldots,d$,
 where each $w_\mu$ is a unitary matrix that commutes with $\Sigma^{(\mu)}$ (and $w_0 = w_d = 1$).}
\end{proposition}
\new{
\begin{corollary}\label{cor:equivprop}
Property (1) in \cref{prop:stre} is equivalent to:
\begin{itemize}
  \item[(2)] It holds
 \begin{align}\label{eq:AequalG}
  A = \mathcal{G}_1 \ck \Sigma^{(1)} \ck \mathcal{G}_2 \ck \Sigma^{(2)} \ck \ldots \ck \mathcal{G}_{(d-2)} \ck \Sigma^{(d-1)} \ck \mathcal{G}_d
 \end{align}
and $\mathcal{G}_1$, $\Sigma^{({\mu}-1)} \mathcal{G}_{\mu}$, ${\mu} = 2,\ldots,d-1$, are left-unitary and $\mathcal{G}_{\mu} \Sigma^{(\mu)}$, ${\mu} = 2,\ldots,d-1$, $\mathcal{G}_d$ are right-unitary\BC{k2}{k} (cf. \cref{def:unfoldings}).
\end{itemize}
Hence also this property provides essential uniqueness.
\end{corollary}
}
%
\begin{proof} \textit{(of \cref{prop:stre})}\BC{l1}{l}\\
\textit{uniqueness:}
 In the following, each $w_\mu$ denotes some unitary matrix that commutes (therefore the lower case letter) with $\Sigma^{(\mu)}$.
 Let there 
 be two such representations $\widetilde{\mathcal{G}}^\sigma$ and $\mathcal{G}^\sigma$. 
 First, $\lhb{\mathcal{G}_1} = \lhb{\widetilde{\mathcal{G}}_1} w_1$ since
 both left-unfoldings contain the left-singular vectors of $A^{(\{1\})}$ due to \cref{eq:sSVD}.
 By induction hypothesis (IH), let $\widetilde{\mathcal{G}}_s = w_{s-1}^H \mathcal{G}_s w_s$ for $s < {\mu}$. 
 Analogously, we have 
 \begin{align*}
 & (\mathcal{G}_1 \ck \ldots \ck \Sigma^{({\mu}-1)} \mathcal{G}_{{\mu}}) \overset{\cref{eq:sSVD}}{=} (\widetilde{\mathcal{G}}_1 \ck \Sigma^{(1)} \widetilde{\mathcal{G}}_2 \ck \ldots \ck \Sigma^{({\mu}-2)} \widetilde{\mathcal{G}}_{{\mu-1}} \ck \Sigma^{({\mu}-1)} \widetilde{\mathcal{G}}_{{\mu}}) w^H_\mu \\
  & \overset{(IH)}{=} ({\mathcal{G}}_1 w_1 \ck \Sigma^{(1)} w_1^H \mathcal{G}_2 w_2 \ck \ldots \ck \Sigma^{(\mu-2)} w_{\mu-2}^H \mathcal{G}_{\mu-1} w_{\mu-1} \ck \Sigma^{({\mu}-1)} \widetilde{\mathcal{G}}_{{\mu}}) w^H_\mu \\  
  & \overset{w_\mu \Sigma^{(\mu)} w_\mu^H = \Sigma^{(\mu)}}{=} {\mathcal{G}}_1 \ck \Sigma^{(1)} \mathcal{G}_2 \ck \ldots \ck \Sigma^{({\mu}-2)} \mathcal{G}_{\mu-1} \ck \Sigma^{(\mu-1)} (w_{\mu-1} \widetilde{\mathcal{G}}_{{\mu}} w^H_\mu) 
 \end{align*}
 Since $T := \mathcal{G}_1 \ck \ldots \ck \Sigma^{({\mu}-2)} \mathcal{G}_{\mu-1}$ is left-unitary by \cref{eq:sSVD},
 the map $H \mapsto T \ck \Sigma^{({\mu}-1)} H$ is injective, and
 it follows $\widetilde{\mathcal{G}}_\mu = w_{\mu-1}^H \mathcal{G}_\mu w_\mu$. This completes the inductive argument. \\\\
 \textit{existence (constructive):}\BC{l2}{l}
 Let $G$ be a representation of $A = \tau_r(G)$, where $G_\mu$, $\mu = 2,\ldots,d$ are right-unitary (this can always be achieved
 using the degrees of freedom within a representation) as well as $V_0 := 1$, $\Sigma^{(0)} := 1$.
 For $\mu = 1,\ldots,d-1$, let the cores $\mathcal{G}_\mu$, $U_\mu$ and the matrix $V_\mu$ be defined via
 \begin{align*}
  \lhb{U_\mu} \ \Sigma^{(\mu)} \ V_\mu^H \overset{\mathrm{SVD}}{:=} \lhb{\Sigma^{(\mu-1)} V_{\mu-1}^H G_\mu}, \quad \mathcal{G}_\mu := (\Sigma^{(\mu-1)})^{-1} U_\mu.
 \end{align*}
 as well as $\mathcal{G}_d := V_{d-1}^H G_d$.
 By construction, \cref{eq:AequalG} holds and each $\Sigma^{(\mu-1)} \mathcal{G}_\mu = U_\mu$ is left-unitary. 
 Since further each $\lhb{U^{\leq \mu}} \Sigma^{(\mu)} \rhb{V_\mu^H G^{>\mu}}$ is an SVD of $A^{(\{1,\ldots,\mu\})}$, $\mu = 1,\ldots,d-1$, also
 \cref{eq:sSVD} holds true.
%
\end{proof}
It is also possible to construct the standard representation directly from $A$ by defining\BC{l3}{l}
  $\lhb{U_\mu} \ \Sigma^{(\mu)} \ B^{(\{1\})}_\mu \overset{\mathrm{SVD}}{:=} B^{(\{1,2\})}_{\mu-1}$, $B_\mu \in \K^{r_\mu \times n_{\mu+1} \times \ldots \times n_d}$, $\mathcal{G}_\mu := (\Sigma^{(\mu-1)})^{-1} U_\mu,$
for $\mu = 1,\ldots,d-1$ as well as $\rhb{\mathcal{G}_d} := B_{d-1}^{(\{1\})}$, and the starting value $B_0^{(\{1,2\})} := A^{(\{1\})}$.
\begin{proof} \textit{(of \cref{cor:equivprop})} 
\textit{``\ $(2) \Rightarrow (1)$'':} Follows directly by transitivity of left- or right-unitary. \\
\textit{``\ $(1) \Rightarrow (2)$'':} 
In the previous construction in the proof of \cref{prop:stre}, the core $\mathcal{G}_{\mu} \Sigma^{({\mu})} = (\Sigma^{(\mu-1)})^{-1} U_\mu \Sigma^{(\mu)} = V_{\mu-1}^H G_\mu V_\mu$
 is right-unitary ($V_0 := 1$, $\Sigma^{(0)} := 1$) and $\Sigma^{(\mu-1)} \mathcal{G}_\mu$ is left-unitary. 
 Due to \cref{prop:stre}, property $(1)$ provides the (essential) uniquenes of that $\mathcal{G}^\sigma$.
 Hence, these constraints hold independently of the construction.
\end{proof}

\old{
\begin{corollary}[Inverse statement]
Let $(\Sigma^{(0)},\mathcal{G}_1,\Sigma^{(1)},\ldots,\mathcal{G}_d,\Sigma^{(d+1)})$ be such that each $\sigma_+^{(\mu)}$ is a positive, weakly
decreasing $r_\mu$-tuple and property $2$ of \repl{Proposition 2.5}{\cref{prop:stre}} is fulfilled. 
Then $A = \tau_r(\mathcal{G})$ is a tensor in $TT(r)$ with TT-singular values $\sigma$.
\end{corollary}
\begin{definition}[Set of weakly decreasing tuples/sequences]
 Let $\D^n_{\geq 0}$, $n \in \N \cup \{\infty\}$, be the set of all weakly decreasing tuples (or sequences) of real numbers with finitely many nonzero elements ($n$ is to be read as index).
 For $n \neq \infty$, the negation $-v \in \D^n$ of $v \in \D^n$ is defined via $-v := (-v_n, \ldots, -v_1)$.
 The positive part $v_{+} \in \D_{> 0}^{\degree(v)}$ is defined as the positive elements of $v$, where $\degree(v) := |\{i \mid v_i > 0\}|$ is its degree.
\end{definition}
}
\begin{corollary}\label{cor:inst}
Let $\mathcal{G}^\sigma = (\mathcal{G}_1,\sigma^{(1)},\mathcal{G}_2,\ldots,\sigma^{(d-1)},\mathcal{G}_d)$\BC{m1}{m} such that property $(2)$ in \cref{cor:equivprop} is fulfilled. 
Then $A$ is a tensor in $TT(r)$ with TT-singular values $\sigma$ and standard representation $\mathcal{G}^\sigma$.
\end{corollary}
\begin{definition}[Set of weakly decreasing tuples/sequences]\label{def:weakdectuple}\BC{b1}{b}\BC{n1}{n}%
 For $n \in \N$, let $\D^n \subset \R^n$ be the cone of weakly decreasing $n$-tuples and let $\D^n_{\geq 0} := \D^n \cap \R^n_{\geq 0}$ be its restriction to non-negative numbers.
 Further, let $\D^\infty_{\geq 0} \subset \R^{\N}$ be the cone of weakly decreasing, non-negative sequences with
 finitely many non-zero entries. \\
 The positive part $v_{+} \in \D_{> 0}^{\degree(v)}$ is defined as the positive elements of $v$, where $\degree(v) := \max_{i : v_i > 0} i$ is its degree.
 For $n \neq \infty$, the negation $-v \in \D^n_{\leq 0}$ of $v \in \D^n_{\geq 0}$ is defined via $-v := (-v_n, \ldots, -v_1)$ (cf. \cite{KnTa01_Hon}).
\end{definition}
For example, for $\gamma = (2,2,1,0,0,\ldots) \in \D^\infty_{\geq 0}$, we have $\degree(\gamma) = 3$ and $\gamma_+ = (2,2,1) \in \D^3_{> 0}$ as well as $-\gamma_+ = (-1,-2-2)$.
Similar to before, we will denote $\Gamma := \diag(\gamma_+) = \left(\begin{smallmatrix} 2 & 0 & 0 \\ 0 & 2 & 0 \\ 0 & 0 & 1 \end{smallmatrix}\right)$.
With a tilde, we will emphasize that a tuple may contain zeros, that is $\widetilde{\gamma} \in \{ v \in \mathcal{D}^n_{\geq 0} \mid v_+ = \gamma_+, \ n \geq \degree(\gamma) \}$\BC{o1}{o}.
For example, we may have $\widetilde{\gamma} = (2,2,1,0) \in \mathcal{D}_{\geq 0}^4$.\\\\
%
%
By basic linear algebra, a left-unitary core $H \in (\K^{1 \times k})^{\{1,\ldots,m\}}$ (analogously a right-unitary core $H \in (\K^{k \times 1})^{\{1,\ldots,m\}}$) exists
if and only if $k \leq m$. In three dimensions, the decoupling through the standard representation hence yields:
%
\begin{corollary}\label{lem:feasofapair}\BC{kc1}{kc}
 For a natural number $m \in \N$, a pair $(\gamma,\theta) \in \D^\infty_{\geq 0} \times \D^\infty_{\geq 0}$ is feasible for the triplet $(\degree(\gamma), m, \degree(\theta))$
 if and only if there exists a core $H \in (\K^{\degree(\gamma) \times \degree(\theta)})^{\{1,\ldots,m\}}$
 for which $\Gamma H$ is left-orthogonal and $H \Theta$ is right-orthogonal.
 \end{corollary}
 \begin{proof}
 Follows directly from \cref{cor:equivprop,cor:inst}. 
 \end{proof}
 \begin{corollary}[Decoupling] \label{cor:pairfeas}\BC{p1}{p}
 $\sigma \in (\D^\infty_{\geq 0})^{d-1}$ is feasible for $n \in \N^d$ if and only
 if $\degree(\sigma^{(1)}) < n_1$, $\degree(\sigma^{(d-1)}) < n_d$ and for each $\mu = 2,\ldots,d-1$, the pair $(\sigma^{(\mu-1)},\sigma^{(\mu)})$ is
 feasible for $(\degree(\sigma^{(\mu-1)}),n_\mu,\degree(\sigma^{(\mu)}))$.
 \end{corollary}
 \begin{proof}
 Follows directly from \cref{cor:inst,cor:equivprop,lem:feasofapair}. 
%
%
%
%
\end{proof}
%
%
%

%
%
\old{\begin{definition}[feasibility]\label{def:fe}
For $\nu < \mu \in \N$, let $\sigma = (\sigma^{(\nu)},\ldots,\sigma^{(\mu)}) \in \D^{\infty}_{\geq 0} \times \ldots \times \D^{\infty}_{\geq 0}$ be a list of weakly decreasing sequences.
Then $\sigma$ is called {\normalfont for $(n_{\nu+1},\ldots,n_{s}) \in \N^{\mu-\nu}$} 
if there exist cores $\mathcal{G}_{\nu+1},\ldots,\mathcal{G}_{\mu}$ with $\mathcal{G}_{s} \in (\R^{r_{{s}-1} \times r_{{s}}})^{n_s}$ (that is, $\mathcal{G}_{s}$ has length $n_s$) for
$r_s := \degree(\sigma^{(s)})$, such that 
$\Sigma^{({s}-1)} \mathcal{G}_{s}$ is left-orthogonal and $\mathcal{G}_{s} \Sigma^{(s)}$ is right-orthogonal, for each ${s} = \nu,\ldots,\mu$. 
\end{definition}
Note that due to \cref{cor:inst}, if $\nu = 0, \mu = d$, $r_0 = r_d = 1$
and $\sigma^{(0)} = \sigma^{(d)} = \|A\|_F$, this coincides with the feasibility of $(\sigma^{(1)},\ldots,\sigma^{(d-1)})$ for a tensor (cf. \cref{def:deffeasibility}).
Using the standard representation, global feasibility can be decoupled into smaller and much simpler problems. }
\old{
\begin{theorem}[Reduction to mode-wise eigenvalue problems]
\hspace*{0cm}
\begin{enumerate}
 \item Let $\mu - \nu > 1$: For each single $h \in \mathbb{N}$, $\nu < h < \mu$, it holds: The list $\sigma$, as in \cref{def:fe}, is feasible for $(n_{\nu+1},\ldots,n_{\mu})$ if and only if 
 $(\sigma^{(\nu)},\ldots,\sigma^{(h)})$ is feasible for $(n_{\nu+1},\ldots,n_{h})$ as well as $(\sigma^{(h)},\ldots,\sigma^{(\mu)})$
 is feasible for $(n_{h+1},\ldots,n_{\mu})$.
 \item Let $\mu = \nu + 1$: A pair $(\gamma,\theta) \in \D^{\infty}_{\geq 0} \times \D^{\infty}_{\geq 0}$ of weakly decreasing sequences is feasible for $n \in \N$ if and only if the following holds: \\
 There exist $n$ pairs of symmetric, positive semi-definite matrices $(A^{(i)},B^{(i)}) \in \R^{\degree(\theta) \times \degree(\theta)} \times \R^{\degree(\gamma) \times \degree(\gamma)}$,
 each with identical (multiplicities of) eigenvalues up to zeros, such that $A := \sum_{i=1}^{n} A^{(i)}$ has eigenvalues $\theta_+^2$ and $B := \sum_{i=1}^{n} B^{(i)}$ has eigenvalues $\gamma_+^2$ .
\end{enumerate}
\end{theorem}}
\new{\begin{theorem}[Equivalence to an eigenvalue problem]\label{thm:fundthm}
     Let $m \in \N$. A pair $(\gamma,\theta) \in \D^{\infty}_{\geq 0} \times \D^{\infty}_{\geq 0}$ is feasible for $(\degree(\gamma),m,\degree(\theta))$ if and only if the following holds: 
 there exist $m$ pairs of Hermitian\footnote{for $\K = \R$, Hermitian is just symmetric and the conjugate transpose $\cdot^H$ is just the transpose $\cdot^T$}, positive semi-definite matrices $(A^{(i)},B^{(i)}) \in \K^{\degree(\theta) \times \degree(\theta)} \times \K^{\degree(\gamma) \times \degree(\gamma)}$,
 each with identical (multiplicities of) eigenvalues up to zeros, such that $A := \sum_{i=1}^{m} A^{(i)}$ has eigenvalues $\theta_+^2$ and $B := \sum_{i=1}^{m} B^{(i)}$ has eigenvalues $\gamma_+^2$ .
     \end{theorem}}
\begin{proof} \textit{(constructive)}%
\old{The first statement is merely transitivity. For $\mu = \nu + 1$, } We show both directions separately.\\ 
 ``$\Rightarrow$'': Let $(\gamma,\theta)$ be feasible for $m$. Then by \old{definition }\new{\cref{lem:feasofapair}}, for $\Gamma = \diag(\gamma_+)$, $\Theta = \diag(\theta_+)$ and a single core $\widehat{N}$, we have both
 $\sum_{i = 1}^m \widehat{N}(i)^H \ \Gamma^2 \ \widehat{N}(i) = I$ as well as $\sum_{i = 1}^m \widehat{N}(i) \ \Theta^2 \ \widehat{N}(i)^H = I$.
 By substitution of $ \widehat{N} = \Gamma^{-1} \ N \ \Theta^{-1}$, this is equivalent to
 \begin{align} 
 \sum_{i = 1}^m N(i)^H \ N(i) = \Theta^2, & \quad \sum_{i = 1}^m \ N(i) \ N(i)^H = \Gamma^2. \label{eq:red}
 \end{align}
 Now, for $A^{(i)} := N(i)^H \ N(i)$ and $B^{(i)} := N(i) \ N(i)^H$, we have found matrices as desired, since the eigenvalues of $A^{(i)}$ and $B^{(i)}$ are
 each the same (up to zeros). \\
 ``$\Leftarrow$'': Let $A^{(i)}$ and $B^{(i)}$ be matrices as required. Then, by eigenvalue decompositions, $A = Q_A \ \Theta^2 \ Q_A^H$, $B = Q_B \ \Gamma^2 \ Q_B^H$ for unitary $Q_A$, $Q_B$ and
 thereby $\sum_{i=1}^{m} Q_A^H \ A^{(i)} \ Q_A = \Theta^2$ and $\sum_{i=1}^{m} Q_B^H \ B^{(i)} \ Q_B = \Gamma^2$. Then again, by truncated eigenvalue decompositions of 
 these summands, we obtain
 \[ Q_A^H \ A^{(i)} \ Q_A = V_i \ S_i \ V_i^H, \quad Q_B^H \ B^{(i)} \ Q_B = U_i \ S_i \ U_i^H, \quad S_i \in \R^{r \times r} \]
 for $r = \min(\degree(\gamma),\degree(\theta))$, unitary (eigenvectors) $V_i, U_i$ and shared (positive eigenvalues) $S_i$. With the choice $N(i) := U_i \ S_i^{1/2} \ V_i^H$, we arrive at \cref{eq:red}, which is equivalent to the desired statement.
\end{proof}

\new{\begin{remark}[Diagonalization]\label{rem:diag}
 Since the condition regarding the sums of Hermitian matrices in \cref{thm:fundthm} remains true under conjugation, we may assume, without loss of generality, that $A = \Theta^2$ and $B = \Gamma^2$. 
\end{remark}}

\section{Feasibility of pairs}\label{sec:feasofpairs}\BC{p2}{p}
\new{We have shown in the previous section, i.e. \cref{cor:pairfeas}, that we only have to consider the feasibility of pairs $(\gamma,\theta)$
for mode sizes $(\degree(\gamma),m,\degree(\theta))$. In order to avoid the redundant entries $\degree(\gamma)$
and $\degree(\theta)$, we will from now on abbreviate as follows:
\begin{definition}[Feasibility of pairs]\label{def:feasofpairs}\BC{w2}{w}
 For $m \in \N$, we say a pair $(\gamma,\theta)$ is feasible for $m$ if and only if it is feasible for $(\degree(\gamma),m,\degree(\theta))$ (cf. \cref{def:deffeasibility}).
\end{definition}
As outlined in \cref{sec:QMPeqTFP}, the property is equivalent to the compatibility of $(\gamma^2,\theta^2)$ for $(\degree(\gamma),m)$ given $\mathcal{K} = \{\{1\},\{1,2\}\}$.\BC{z6}{z}
In fact, there exist several results on this topic as discussed in \cref{sec:QMC}, e.g. that compatible pairs form a cone. In the following, we analyze the problem from the 
different perspective provided by \cref{thm:fundthm}. }

\subsection{Constructive, diagonal feasibility}\label{sec:diagonalfeas}

The feasibility of pairs is a reflexive and symmetric relation, but it is not transitive.
In some cases, verification can be easier:\BC{kf1}{kf}

\begin{lemma}[Diagonally feasible pairs]\label{lem:trivial}
Let $(\gamma,\theta) \in \mathcal{D}^{\infty}_{\geq 0} \times \mathcal{D}^{\infty}_{\geq 0}$ as well as
$a^{(1)},\ldots,a^{(\new{m})} \in \R^r_{\geq 0}$, $r = \mathrm{max}(\degree(\gamma),\degree(\theta))$, and permutations $\pi_1,\ldots,\pi_{\new{m}} \in S_r$ such that
\[ a^{(1)}_i + \ldots + a^{(\new{m})}_i = \gamma_i^2, \quad a^{(1)}_{\pi_1(i)} + \ldots + a^{(\new{m})}_{\pi_{\new{m}}(i)} = \theta_i^2, \quad i = 1,\ldots,r. \]
Then $(\gamma,\theta)$ is feasible for $\new{m}$ (we write \textit{diagonally feasible} in that case).
For $\new{m}\new{, r_1, r_2 \in \N}$, $\gamma_+^2 = (1,\ldots,1)$ of length $r_1$ and $\theta_+^2 = (k_1,\ldots,k_{r_2}) \in \D^{r_2}_{\new{\geq 0}} \cap \{1,\ldots,\new{m}\}^{r_2}$, with $\|k\|_1 = r_1$,
the pair $(\gamma,\theta)$ is diagonally feasible for $\new{m}$.
\end{lemma}
\begin{proof}
The given criterion is just the restriction to diagonal matrices in \cref{thm:fundthm}. 
All sums of zero-eigenvalues can be ignored, i.e. we also find
diagonal matrices of actual sizes $\degree(\gamma) \times \degree(\gamma)$ and $\degree(\theta) \times \degree(\theta)$.
The subsequent explicit set of feasible pairs follows immediately by restricting $a^{(\ell)}_i \in \{0,1\}$ and by using appropriate permutations.
\end{proof}
For example, to show that $(\gamma,\theta)$, $\gamma_+^2 = (1,1,1,1)$, $\theta_+^2 = (2,2)$, is feasible for $m = 2$, we can set
$a^{(1)} = (1,1,0,0)$, $a^{(2)} = (0,0,1,1)$ and $\pi_1 = \mathrm{Id}$, $\pi_2 = (1,3,2,4)$. The resulting matrices in \cref{thm:fundthm}
then are $B^{(1)} = \mathrm{diag}((1,1,0,0))$, $B^{(2)} = \mathrm{diag}((0,0,1,1))$ as well as $A^{(1)} = A^{(2)} = \mathrm{diag}((1,1))$.
Following the procedure in \cref{thm:fundthm}, we obtain the single core $\new{N}$, for which $\Gamma \new{N}$, $\new{N} \Theta$ are left- and right-unitary, respectively:
\begin{align*}
 \new{N}(1) = \left[\begin{smallmatrix}
                         \frac{1}{\sqrt{2}}  &       0 \\
         0  &  \frac{1}{\sqrt{2}}\\
         0  &       0\\
         0  &       0
                      \end{smallmatrix}\right], \quad
                       \new{N}(2) = \left[\begin{smallmatrix}
         0  &       0\\
         0  &       0\\
                                  \frac{1}{\sqrt{2}}  &       0 \\
         0  &  \frac{1}{\sqrt{2}}
                      \end{smallmatrix}\right].
\end{align*}
Although for $\new{m} = 2$, $r \leq 3$, each feasible pair happens to be diagonally feasible, this does not hold in general.
For example, the pair $(\gamma,\theta)$,
\begin{align} \gamma_+^2 = (7.5,5) \quad \mbox{and} \quad \theta^2_+ = (6,3.5,2,1), \label{eq:not_triv_pair} \end{align} 
is feasible (cf. \cref{eq:n2leqn1} or \cref{fig:plot_not_triv_feasible}) for $\new{m} = 2$, but it is quite easy to verify that it
is not diagonally feasible. 

\begin{definition}[Set of feasible pairs]\label{def:setoffeaspairs}
Let $\mathcal{F}_{\new{m},(r_1,r_2)}$ be the set of pairs $(\widetilde{\gamma},\widetilde{\theta}) \in \mathcal{D}^{r_1}_{\geq 0} \times \mathcal{D}^{r_2}_{\geq 0}$,
for which $(\gamma,\theta) = ((\widetilde{\gamma},0,\ldots),(\widetilde{\theta},0,\ldots))$ is feasible for $\new{m}$ (cf. \cref{def:feasofpairs}), and
\[ \mathcal{F}^2_{\new{m},(r_1,r_2)} := \{ (\gamma_1^2, \ldots, \gamma_{r_1}^2, \theta_1^2, \ldots, \theta_{r_2}^2) \mid (\widetilde{\gamma},\widetilde{\theta}) \in \mathcal{F}_{\new{m},(r_1,r_2)} \}. \]
\end{definition}%
\new{The following theorem is a special case of \cref{eq:n2leqn1} and features a constructive proof as outlined below.}%
\begin{theorem}\label{thm:prop1}
Let $m\in \N$. If $r_1, r_2 \leq \new{m}$, then 
\[ \mathcal{F}_{\new{m},(r_1,r_2)} = \mathcal{D}^{r_1}_{\geq 0} \times \mathcal{D}^{r_2}_{\geq 0} \cap \{(\widetilde{\gamma},\widetilde{\theta}) \mid \|\widetilde{\gamma}\|_2 = \|\widetilde{\theta}\|_2\},
\]
that is, any pair $(\gamma,\theta) \in \D^{\infty}_{\geq 0} \times \D^{\infty}_{\geq 0}$ with $\degree(\gamma), \degree(\theta) \leq \new{m}$, for which the trace property holds true, is (diagonally) feasible for $\new{m}$.
\end{theorem}
\begin{proof}
We give a proof by contradiction. Set $\widetilde{\gamma} = (\gamma_+,0,\ldots,0)$ as well as
$\widetilde{\theta} = (\theta_+,0,\ldots,0)$ such that both have length $\new{m}$. 
Let the permutation $\widetilde{\pi}$ be given by the cycle $(1,\ldots,\new{m})$ and $\pi_\ell := \widetilde{\pi}^{\ell-1}$.
For each $k$, let $R_k := \{ (i,\ell) \mid \pi_\ell(k) = i \}$.  
Now, let the nonnegative (eigen-) values $a^{(\ell)}_i$, \ $\ell,i = 1,\ldots,\new{m}$, form a minimizer
of $w := \|A (1,\ldots,1)^T - \widetilde{\gamma}^2\|_1$, subject to
\[ \sum_{(i,\ell) \in R_k} a^{(\ell)}_i = a^{(1)}_{\pi_1(k)} + \ldots + a_{\pi_{\new{m}}(k)}^{(\new{m})} = \theta_k^2, \ k = 1,\ldots,\new{m}, \]
where $A = \{a^{(\ell)}_i\}_{(i,\ell)}$ (the minimizer exists since the allowed values form a compact set). For $\new{m} = 3$, for example, 
we aim at the following, where $R_3$ has been highlighted.
\begin{align*}
 \begin{pmatrix}
  \greym{a^{(1)}_{\pi_1(1)}} & a^{(2)}_{\pi_2(3)} & \greym{a^{(3)}_{\pi_3(2)}} \\
  \greym{a^{(1)}_{\pi_1(2)}} & \greym{a^{(2)}_{\pi_2(1)}} & a^{(3)}_{\pi_3(3)} \\
  a^{(1)}_{\pi_1(3)} & \greym{a^{(2)}_{\pi_2(2)}} & \greym{a^{(3)}_{\pi_3(1)}}
 \end{pmatrix}
  \cdot \begin{pmatrix}
    1 \\
    1 \\
    1
   \end{pmatrix}
 \rightarrow \begin{pmatrix}
    \gamma_1^2 \\
    \gamma_2^2 \\
    \gamma_3^2
   \end{pmatrix}
\end{align*}
Let further
\[ \#_{\gtrless} := \old{|}\{ i \mid a^{(1)}_i + \ldots + a_i^{(\new{m})} \gtrless \gamma_i^2, \ i = 1,\ldots,\new{m} \}\old{|}. \]
As $\|\gamma\|_2 = \|\theta\|_2$ by assumption, either \new{$\#_>$ and $\#_<$ are both empty or both not empty}. In the first case, we are finished.
Assume therefore there is an $(i,j) \in \#_> \times \#_<$. Then there is an index $\ell_1$ such that
$a^{(\ell_1)}_i > 0$ as well as indices $k$ and $\ell_2$ such that $(i,\ell_1), (j,\ell_2) \in R_k$.
This is however a contradiction, since replacing $a^{(\ell_1)}_i \leftarrow a^{(\ell_1)}_i - \varepsilon$ 
and $a^{(\ell_2)}_j \leftarrow a^{(\ell_2)}_j + \varepsilon$ for some small enough $\varepsilon > 0$ 
is valid, but yields a lower minimum $w$. Hence it already holds $a^{(1)}_i + \ldots + a^{(\new{m})}_i = \gamma_i^2, \quad i = 1,\ldots,\new{m}$.
Due to \cref{lem:trivial}, the pair $(\gamma,\theta)$ is feasible.
\end{proof}
\old{Note that although the proof is by contradiction, }%
The entries\BC{kd1}{kd} $a^{(\ell)}_i$ can be found via a linear programming algorithm, since they are given through linear constraints.
%
A corresponding core can easily be calculated subsequently,
as the proof of \cref{thm:fundthm} is constructive.\\
In the following section, we address 
theory that was subject to a near century long development. Fortunately, 
many results in that area can be transferred ---
last but not least because of the work of A. Knutson and T. Tao
and their \new{illustrative }theory of \textit{honeycombs} \cite{KnTa01_Hon}. \old{It should be noted that with the connection
made through \repl{Theorem 2.14}{\cref{thm:fundthm}}, a complete resolution of the feasibility problem
would establish a vast part of theory for eigenvalues of Hermitian matrices.
Hence at this point, we can rather expect to find answers in the latter area than
the other way around.}

\subsection{Weyl's Problem and the Horn Conjecture}\label{sec:weylhorn}

In 1912, H. Weyl posed a problem \cite{We1912_Das} that asks for an analysis of the following relation.

\begin{definition}[Eigenvalues of a sum of two Hermitian matrices \cite{KnTa01_Hon}]\label{def:sumsofhermi}
Let $\lambda, \mu, \nu \in \mathcal{D}^n$. Then the relation
\begin{align}
 \lambda \boxplus \mu \sim_c \nu \label{eq:simc}
\end{align}
is defined to hold if there exist Hermitian matrices $A,B \in \C^{n \times n}$ and $C := A + B$ with 
eigenvalues $\lambda, \mu$ and $\nu$, respectively.
This definition is straight forwardly extended to more than two summands.\footnote{The symbol $\boxplus$ used in \cite{KnTa01_Hon} only appears within such relations and hints at the addition of $A$ and $B$. 
There is no relation to the earlier used $\boxtimes$.}
\end{definition}
The relation \cref{eq:simc} may\BC{ke1}{ke} equivalently be written as $\lambda \boxplus \mu \boxplus (-\nu) \sim_c 0$ (cf. \cite{KnTa01_Hon}, \cref{def:weakdectuple}).
A result which was discovered much later by Fulton \cite{Fu00_Eig}, which we want to pull forward, states
that there is no difference when restricting oneself to real matrices.%
\begin{theorem}[{\cite[Theorem 3]{Fu00_Eig}}]\label{thm:fultonsym}
 A triplet ($\lambda,\mu,\nu)$ occurs as eigenvalues for an associated triplet of real symmetric matrices
 if and only if it appears as one for Hermitian matrices.
\end{theorem}
\new{
Assuming without loss of generality $\degree(\gamma) \leq \degree(\theta)$, the condition (cf. \cref{thm:fundthm}) for the feasibility of a pair $(\gamma,\theta)$ for $m$ can now be restated as: 
there exist $a_1,\ldots,a_m \in \mathcal{D}_{\geq 0}^{\degree(\gamma)}$ with $a_1 \boxplus \ldots \boxplus a_m \sim_c \gamma_+^2$
and $(a_1,0,\ldots) \boxplus \ldots \boxplus (a_m,0,\ldots) \sim_c \theta_+^2$.
The later \cref{thm:relaxed} uses \cref{thm:fultonsym} to confirm that the initial choice $\K \in \{\R,\C\}$ is also irrelevant regarding the conditions for feasibility.
\\\\}
Weyl and Ky Fan \cite{Fa1949_Ona} were among the first ones to give necessary, linear inequalities to the relation \cref{eq:simc}. %
We refer to the (survey) article 
\textit{Honeycombs and Sums of Hermitian Matrices}\footnote{To the best of our knowledge, in Conjecture 1 (Horn conjecture) on page 176 of the AMS publication,
the relation $\geq$ needs to be replaced by $\leq$. This is a mere typo without any consequences and the authors are most
likely aware of it by now.}
\cite{KnTa01_Hon} by Knutson and Tao, which has been
the main point of reference for the remaining part and serves as historical survey as well (see also \cite{Bh01_Lin}). 
We use parts of their notation as long as we remain within this topic.\BC{kf2}{kf}
Therefor, $m$ remains the number of matrices ($m=2$ in \cref{def:sumsofhermi}), but $n$ denotes the size of the Hermitian matrices and $r$ is used as index.
A. Horn introduced the famous \textit{Horn conjecture} in 1962:


\begin{theorem}[(Verified) Horn conjecture \cite{Ho1962_Eig}]\label{thm:hornconj}
There is a specific set $T_{r,n}$ (defined for example in \cite{Bh01_Lin}) of triplets of monotonically increasing $r$-tuples such that:
The relation $\lambda \boxplus \mu \sim_c \nu$ is satisfied if and only
if for each $(i,j,k) \in T_{r,n}, \ r = 1,\ldots,n-1$, the inequality
\begin{align} \nu_{k_1} + \ldots + \nu_{k_r} \leq \lambda_{i_1} + \ldots + \lambda_{i_r} + \mu_{j_1} + \ldots + \mu_{j_r} \label{eq:hornineq} \end{align}
holds, as well as the trace property $\sum_{i=1}^n \lambda_i + \sum_{i=1}^n \mu_i = \sum_{i=1}^n \nu_i$.
\end{theorem}

As already indicated, the conjecture is correct, as proven \new{through the contributions of }Knutson and Tao (cf. \cref{sec:honeycomb}) 
and Klyachko (\cite{Kl98_Sta}).
Fascinatingly, the quite inaccessible, recursively defined set $T_{r,n}$ can in turn be described
by eigenvalue relations themselves, as stated by W. Fulton \cite{Fu00_Eig}.

\begin{theorem}[Description of $T_{r,n}$ \cite{Fu00_Eig,Ho1962_Eig,KnTa01_Hon}]\label{thm:desTrn}
Let $\bigtriangleup \ell:=(\ell_r - r,\ell_{r-1}-(r-1),\ldots,\ell_2-2,\ell_1-1)\new{ \in \mathcal{D}^r_{\geq 0}}$ for any set or 
tuple $\ell$ of $r$ increasing natural numbers. 
The triplet $(i,j,k)$ of such is in $T_{r,n}$ if and only if for the
corresponding triplet it holds $\bigtriangleup i \boxplus \bigtriangleup j \sim_c \bigtriangleup k$.
\end{theorem}

Even with just diagonal matrices, one can thereby derive various (possibly all) triplets. For example,
Ky Fan's inequality \cite{Fa1949_Ona},
$\sum_{i=1}^k \nu_i \leq \sum_{i=1}^k \lambda_i + \sum_{i=1}^k \mu_i$,
relates to the simple $0 \boxplus 0 \sim_c 0 \in \R^k$, $k = 1,\ldots,n$. A further
interesting property, as already shown by Horn, is given if \cref{eq:hornineq} holds as equality:

\begin{lemma}[\cite{Ho1962_Eig,KnTa01_Hon}]\label{lem:reduc}
Let $(i,j,k) \in T_{r,n}$ and $\lambda \boxplus \mu \sim_c \nu$. Further, let $i^c, j^c, k^c$ be
their complementary indices with respect to $\{1,\ldots,n\}$. Then the following statements are equivalent:
\begin{itemize}
 \item $\nu_{i_1} + \ldots + \nu_{i_r} = \lambda_{i_1} + \ldots + \lambda_{i_r} + \mu_{j_1} + \ldots + \mu_{j_r}$
 \item Any associated triplet of Hermitian matrices $(A,B,C)$ is block diagonalizable into two parts, which
 contain eigenvalues indexed by $(i,j,k)$ and $(i^c, j^c, k^c)$, respectively.
 \item $\lambda|_i \boxplus \mu|_j \sim_c \nu|_k$
 \item $\lambda|_{i^c} \boxplus \mu|_{j^c} \sim_c \nu|_{k^c}$
 \end{itemize}
\end{lemma}
The relation is in that sense split into two with respect to the triplet $(i,j,k)$.

%



\section{Honeycombs and hives}\label{sec:honeycomb}

The following result by Knutson and Tao poses a complete resolution to Weyl's problem
and is based on preceding breakthroughs \cite{HeRo95_Eig, Kl98_Sta, KnTa99_The, KnTaWo04_The}.
This problem has since then also been generalized, for example \cite{Friedland00_Fin,Fulton00_Eig}.

\subsection{\new{Honeycombs and eigenvalues of sums of hermitian matrices}}
While we can only give a quick introduction, the article \cite{KnTa01_Hon} provides a good understanding of \textit{honeycombs} --- \old{the }\new{a }central tool in 
the verification of the Horn conjecture.
They \old{are designed to exactly reflect the mathematics behind the relation $\sim_c$
and }allow graph theory as well as linear programming to be applied to Weyl's problem.
A honeycombs $h$ (cf. \cref{fig:archetype}) \old{can be described as }\new{is a }two dimensional object, embedded into 
$h \subset \R^3_{\sum = 0} := \{ x \in \R^3 \mid x_1 + x_2 + x_3 = 0\}$,
consisting of line segments (edges or rays), each parallel to one of the cardinal directions $(0,1,-1)$ (north west), $(-1,0,1)$ (north east) or $(1,-1,0)$ (south),
as well as vertices, where those join. Thereby, each segment has exactly one constant coordinate, the collection of which
we formally denote with $\EDGE(h) \in \mathbb{R}^N$, $N = \frac{3}{2} n(n+1)$ (including the boundary rays).
Non-degenerate $n$-honeycombs follow one identical topological structure
and are identifiable through linear constraints:
the constant coordinates of three edges meeting at a vertex add up to zero, and every edge has strictly positive length. 
This leads
to one archetype, as displayed in \cref{fig:archetype} (for $n = 3$). 
The involved eigenvalues appear as \textit{boundary values} $\delta(h) := (\mathfrak{w}(h),\mathfrak{e}(h),\mathfrak{s}(h)) := (\lambda,\mu,-\nu) \in (\mathcal{D}^n)^3$ (west, east and south),
i.e. the constant coordinates of the outer rays. 

\begin{figure}[tbhp]
  \centering
  \ifuseprecompiled
\includegraphics{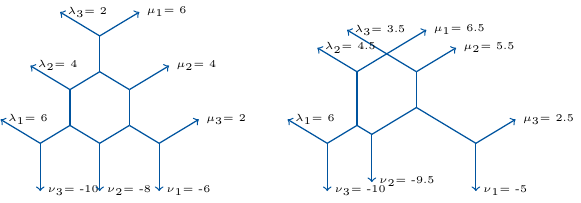}
\else
\setlength\figureheight{3cm}
\setlength\figurewidth{9cm}
\tikzsetnextfilename{plot_archetype}
\input{tikz_base_files/plot_archetype/plot_archetype.tikz}
\fi
\caption{\label{fig:archetype} Left: The archetype of non-degenerate $(n=3)$-honeycombs as described in \cref{sec:honeycomb}.
The rays pointing in directions north west, north east and south have constant coordinates $\mathfrak{w}(h)_i = \lambda_i$, $\mathfrak{e}(h)_i = \mu_i$ and $\mathfrak{s}(h)_i = -\nu_i$, respectively. 
The remaining line segments contribute to the total \textit{edge length} of the honeycomb. Right: A degenerate honeycomb,
where the line segment at the top has been completely contracted. Here, only eight line segments remain to contribute to the total edge length.
}
\end{figure}
The set $\HONEY_n$ of all $n$-honeycombs is identified as the closure of the set of non-degenerate ones, allowing edges of length zero as well.
Thereby, $C = \{ \EDGE(h) \mid h \in \HONEY_n \} \subset \mathbb{R}^N$ is a closed, convex, polyhedral cone.

\begin{theorem}[\new{Relation to honeycombs }\cite{KnTa01_Hon}]
The relation $\lambda \boxplus \mu \sim_c \nu$ is satisfied if and only if
there exists a honeycomb $h$ with boundary values $\delta(h) = (\lambda, \mu, -\nu)$.
\end{theorem}
The set of triplets $\{ (\lambda,\mu,-\nu) \new{\in (\mathcal{D}^n)^3 }\mid \lambda \boxplus \mu \sim_c \nu\}$ thus equals
$\BDRY_n := \{ \delta(h) \mid h \in \HONEY_n \}$, which is at the same time the orthogonal projection of the cone $C$ to the
coordinates associated with the boundary (the rays) --- and, as shown in its verification, the very same cone described by the (in)equalities in 
\cref{thm:hornconj}.\\
There is also a related statement implicated by the ones in \cref{lem:reduc}.
If a triplet $(i,j,k) \in T_{r,n}$ yields an equality as in \cref{eq:hornineq}, then for the associated honeycomb $h$, $\delta(h) = (\lambda,\mu,-\nu)$,
it holds
\begin{align}
 h = h_1 \otimes h_2, \quad \delta(h_1) = (\lambda|_i, \mu|_j, -\nu|_k), \ \delta(h_2) = (\lambda|_{i^c}, \mu|_{j^c}, -\nu|_{k^c}) \label{eq:overlay},
\end{align}
which means that $h$ is a literal overlay of two smaller honeycombs. Vice versa, if a
honeycomb is an overlay of two smaller ones, then it yields two separate eigenvalue relations, however
the splitting does not necessarily correspond to a triplet in $T_{r,n}$ \cite{KnTa01_Hon}.
%
%
\subsection{\new{Hives and feasibility of pairs}}
\begin{definition}[Positive semi-definite honeycomb]
We define a positive semi-definite honeycomb $h$ as a honeycomb with boundary values $\mathfrak{w}(h), \mathfrak{e}(h) \geq 0$ and $\mathfrak{s}(h) \leq 0$. 
\end{definition}
%
A honeycomb can connect three matrices. In order to connect $m$ matrices,
chains or systems of honeycombs are put in relation to each other through their
boundary values. Although the phrase \textit{hive} has appeared before as similar object to honeycombs,
to which we do not intend to refer here, we use it to emphasize that a collection of honeycombs is given\footnote{in absence of further \textit{bee} related vocabulary}.
Considerations for simple chains of honeycombs (cf. \cref{lem:sevherm}) have also been made in \cite{KnTa99_The, KnTaWo04_The}, but we need to rephrase these
ideas for our own purposes.
\new{\BC{kg1}{kg}
\begin{definition}[Hives]\label{def:hive}
Let $n,M \in \N$. We define a (pos. semi-definite) $(n,M)$-hive $H$ as a collection of $M$ (pos. semi-definite) $n$-honeycombs $h^{(1)},\ldots,h^{(M)}$.
\end{definition}
\begin{definition}[Structure of hives]
Let $H$ be an $(n,M)$-hive and $B := \{ (i,\mathfrak{b}) \mid i = 1,\ldots,M, \ \mathfrak{b} \in \{\mathfrak{w},\mathfrak{e},\mathfrak{s}\} \}$.
Further, let $\sim_S\ \in B \times B $ be an equivalence relation. We say $H$ has structure $\sim_S$ if the following holds: \\
Provided $(i,\mathfrak{b}) \sim_S (j,\mathfrak{p})$, then if both $\mathfrak{b}$ and $\mathfrak{p}$ or neither of them equal $\mathfrak{s}$, it holds $\mathfrak{b}(h^{(i)}) = \mathfrak{p}(h^{(j)})$, or
otherwise $\mathfrak{b}(h^{(i)}) = -\mathfrak{p}(h^{(j)})$.\\\\%
We define the hive set $\HIVE_{n,M}(\sim_S)$ as set of all $(n,M)$-hives $H$ with structure $\sim_S$.
\end{definition}
In order to specify a structure $\sim_S$, we will only list generating sets of equivalences (with respect to reflexivity, symmetry and transitivity).
\begin{definition}[Boundary map of structured hives]
Let $H$ be an $(n,M)$-hive with structure $\sim_S$. Let further $P := \{ (i,\mathfrak{b}) \mid |[(i,\mathfrak{b})]_{\sim_S}| = 1 \}$ be the set of singletons.
We define the boundary map $\delta_P: \HIVE_{n,M}(\sim_S) \rightarrow (\mathcal{D}^n)^P$ to map
any hive $H \in \HIVE_{n,M}(\sim_S)$ to the function $f_P: P \rightarrow \mathcal{D}^n$ defined via: \\ 
For all $(i,\mathfrak{b}) \in P$, if $\mathfrak{b}$ equals $\mathfrak{s}$, it holds
$f_P(i,\mathfrak{b}) = -\mathfrak{b}(h^{(i)})$, or otherwise $f_P(i,\mathfrak{b}) = \mathfrak{b}(h^{(i)})$.
\end{definition}
}

A single $n$-honeycomb $h$ with boundary values $(\lambda,\mu,-\nu)$ can hence be identified as $(n,1)$-hive $H$ with trivial structure $\sim_S$ \new{generated by the empty set},
singleton set $P = \{(1,\mathfrak{w}),(1,\mathfrak{e}),(1,\mathfrak{s})\}$ and
boundary \old{function }$\delta_P(H) = \{(1,\mathfrak{w})\mapsto \lambda,(1,\mathfrak{e})\mapsto \mu,(1,\mathfrak{s})\mapsto \nu\}$%
\footnote{this denotes $f_P(1,\mathfrak{w}) = \lambda$, $f_P(1,\mathfrak{e}) = \mu$, $f_P(1,\mathfrak{s}) = \nu$ for $f_P = \delta_P(H)$}. 
In this sense, it holds $\HONEY_n \cong \HIVE_{n,1}(\emptyset)$ and we regard honeycombs as hives as well.
\new{Another example\BC{kg2}{kg} is illustrated in \cref{fig:hive_example1}, where $\sim_S$ is generated by $(1,\mathfrak{s}) \sim_S (2,\mathfrak{w})$
and $(2,\mathfrak{s}) \sim_S (3,\mathfrak{w})$, such that the singeltons are $P = \{ (1,\mathfrak{w}), (1,\mathfrak{e}), (2,\mathfrak{e}), (3,\mathfrak{e}), (3,\mathfrak{s}) \}$.}

\begin{lemma}[Eigenvalues of a sums of matrices]\label{lem:sevherm}
The relation \\ $a^{(1)} \boxplus \ldots \boxplus a^{(m)} \sim_c c$ is satisfied if and only if
there exists a hive $H$ of size $M = m-1$ (cf. \cref{fig:hive_example1}) with structure $\sim_S$, generated by $(i,\mathfrak{s}) \sim_S (i+1,\mathfrak{w})$, $i = 1,\ldots,M-1$,
and $\delta_P(H) = \{(1,\mathfrak{w})\mapsto a^{(1)},(1,\mathfrak{e}) \mapsto a^{(2)} ,(2,\mathfrak{e}) \mapsto a^{(3)},\ldots,(M,\mathfrak{e}) \mapsto a^{(m)},(M,\mathfrak{s}) \mapsto c\}$.
\end{lemma}
\begin{proof}
 ``$\Rightarrow$'': The relation $a^{(1)} \boxplus \ldots \boxplus a^{(m)} \sim_c c$ is equivalent to the existence of Hermitian (or real symmetric, cf. \cref{thm:fultonsym}) matrices $A^{(1)},\ldots,A^{(m)}$, $C = A^{(1)} + \ldots + A^{(m)}$ with eigenvalues $a^{(1)},\ldots,a^{(m)},c$, respectively.
For $A^{(1,\ldots,k+1)} := A^{(1,\ldots,k)} + A^{(k+1)}, \ k = 1,\ldots,m-1$, with accordant eigenvalues $a^{(1,\ldots,k)}$, the relation can equivalently be
restated as $a^{(1,\ldots,k)} \boxplus a^{(k+1)} \sim_c a^{(1,\ldots,k+1)}, \ k = 1,\ldots,m-1$. This in turn is equivalent to the existence of honeycombs $h^{(1)},\ldots,h^{(m-1)}$ with
boundary values $\delta(h^{(1)}) = (a^{(1)},a^{(2)},-a^{(1,2)}), \delta(h^{(2)}) = (a^{(1,2)},a^{(3)},-a^{(1,2,3)}), \ldots$, $\delta(h^{(m-1)}) = (a^{(1,\ldots,m-1)},a^{(m)},-c)$.
This depicts the structure $\sim_S$ and boundary function $\delta_P(H)$. \\
 ``$\Leftarrow$'': If in reverse the hive $H$ is assumed to exist, then we know, via the single honeycombs,
that there exist matrices $\widetilde{A}^{(1,\ldots,k+1)} = A^{(1,\ldots,k)} + A^{(k+1)}, \ k = 1,\ldots,m-1$ with corresponding eigenvalues.
Although we only know that $\widetilde{A}^{(1,\ldots,k+1)}$ and $A^{(1,\ldots,k+1)}$ share eigenvalues, the remaining, reverse construction
is done via an inductive diagonalization argument (cf. \cref{rem:diag}).
%
\end{proof}
\begin{figure}[tbhp]
  \centering
  \ifuseprecompiled
\includegraphics{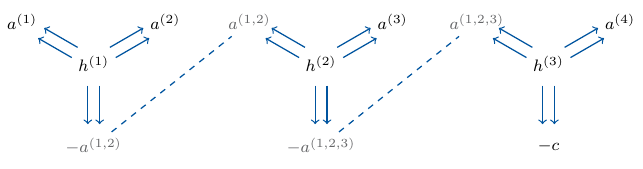}
\else
\setlength\figureheight{2cm}
\setlength\figurewidth{14cm}
\tikzsetnextfilename{hive_example1}
\input{tikz_base_files/hive_example1/hive_example1.tikz}
\fi
\caption{\label{fig:hive_example1} The schematic display of an $(n,3)$-hive $H$ with structure $\sim_S$ as in \cref{lem:sevherm}. 
North west, north east and south rays correspond to the boundary values $\mathfrak{w}(h_i)$, $\mathfrak{e}(h_i)$ and $\mathfrak{s}(h_i)$, respectively.
Coupled boundaries are in gray and connected by dashed lines.}
\end{figure}
The idea behind honeycomb overlays (cf. \cref{eq:overlay}) can be extended to hives as well:
\begin{lemma}[Zero eigenvalues]\label{lem:zeroeig}
If the relation \\ $a^{(1)} \boxplus \ldots \boxplus a^{(m)} \sim_c c$ is satisfied for $a^{(i)} \in \mathcal{D}_{\geq 0}^n$, $i = 1,\ldots,m$,
and $c_n = 0$, then $a^{(1)}_n = \ldots = a^{(m)}_n = 0$ and already $a^{(1)}|_{\{1,\ldots,n-1\}} \boxplus \ldots \boxplus a^{(m)}|_{\{1,\ldots,n-1\}} \sim_c c|_{\{1,\ldots,n-1\}}$.
\end{lemma}
\begin{proof}
The first statement follows by basic linear algebra, since $a^{(1)},\ldots,a^{(m)}$ are nonnegative. For the second part,
\cref{lem:sevherm} and \cref{eq:overlay} are used. Inductively, in each
honeycomb of the corresponding hive $H$, a separate $1$-honeycomb with boundary values $(0,0,0)$ can be found. Hence, each honeycomb is an overlay
of such a $1$-honeycomb and an $(n-1)$-honeycomb. All remaining $(n-1)$-honeycombs then form a new hive with identical structure $\sim_S$.
\end{proof}
We arrive at an extended version of \cref{thm:fundthm}.
\begin{theorem}[\new{Equivalence to existence of a hive}]\label{thm:relaxed}
Let $(\gamma,\theta) \in \mathcal{D}^{\infty}_{\geq 0} \times \mathcal{D}^{\infty}_{\geq 0}$ and $n \geq \degree(\gamma), \degree(\theta)$.
Further, let $\widetilde{\theta} = (\theta_+,0,\ldots,0)$, $\widetilde{\gamma} = (\gamma_+,0,\ldots,0)$ be $n$-tuples. 
The following statements are equivalent, independent of the choice $\K \in \{\R,\C\}$:\BC{q0}{q}
\begin{itemize}
 \item The pair $(\gamma,\theta)$ is feasible for $m \in \N$ 
 \item There are $m$ pairs of Hermitian, positive semi-definite matrices $(A^{(i)},B^{(i)}) \in \C^{n \times n} \times \C^{n \times n}$,
 each with identical (multiplicities of) eigenvalues, such that $A := \sum_{i=1}^{m} A^{(i)}$ has eigenvalues $\widetilde{\theta}^2$ and $B := \sum_{i=1}^{m} B^{(i)}$ 
 has eigenvalues $\widetilde{\gamma}^2$, respectively.
 \item There exist $a^{(1)},\ldots,a^{(m)} \in \D_{\geq 0}^n$ such that $a^{(1)} \boxplus \ldots \boxplus a^{(m)} \sim_c \widetilde{\gamma}^2$
 as well as $a^{(1)} \boxplus \ldots \boxplus a^{(m)} \sim_c \widetilde{\theta}^2$. 
 \item There exists a positive semi-definite $(n,M)$-hive $H$ of size $M = 2(m-1)$ (cf. \cref{fig:hive_example2}) with structure $\sim_S$,
where $(i + u,\new{\mathfrak{s}}) \sim_S (i+1 + u,\new{\mathfrak{w}})$, $i = 1,\ldots,M/2-1$, $u \in \{0,M/2\}$, as well as
$(1,\new{\mathfrak{w}}) \sim_S (1+M/2,\new{\mathfrak{w}})$ and $(i,\new{\mathfrak{e}}) \sim_S (i+M/2,\new{\mathfrak{e}})$, $i = 1,\ldots,M$. Further, 
$\delta_P(H) = \{(M/2,\new{\mathfrak{s}}) \mapsto \widetilde{\gamma}^2,(M,\new{\mathfrak{s}}) \mapsto \widetilde{\theta}^2\}$.
\end{itemize}
\end{theorem}
\begin{proof}
The existence of matrices with actual size $\degree(\gamma)$, $\degree(\theta)$, respectively, follows by repeated application of \cref{lem:zeroeig}.
The hive essentially consists of two rows of honeycombs as in \cref{lem:sevherm}.
Therefor, the same argumentation holds, but instead of prescribed boundary values $a^{(i)}$, these values
are coupled between the two hive parts. Due to \cref{thm:fultonsym}, there is no difference whether we consider real or complex matrices and tensors.
\end{proof}
\old{The irrelevance of zeros hence also translates into this setting.}
\begin{figure}[tbhp]
  \centering
  \ifuseprecompiled
\includegraphics{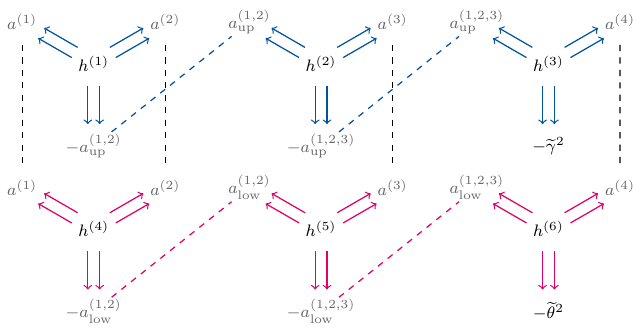}
\else
\setlength\figureheight{2cm}
\setlength\figurewidth{14cm}
\tikzsetnextfilename{hive_example2}
\input{tikz_base_files/hive_example2/hive_example2.tikz}
\fi
\caption{\label{fig:hive_example2} The schematic display of an $(n,6)$-hive $H$ (upper part in blue, lower part in magenta)
with structure $\sim_S$ as in \cref{lem:sevherm}. 
North west, north east and south rays correspond to the boundary values $\mathfrak{w}(h_i)$, $\mathfrak{e}(h_i)$ and $\mathfrak{s}(h_i)$, respectively.
Coupled boundaries are in gray and connected by dashed lines.}
\end{figure}
The feasibility of $(\gamma,\theta)$ as in \cref{eq:not_triv_pair} is 
provided by the hive in \cref{fig:plot_not_triv_feasible}. Even though not diagonally feasible,
the pair can be disassembled, as later shown in \cref{sec:vdescr}, into multiple, diagonally
feasible pairs, which then as well prove its feasibility.
\begin{figure}[tbhp]
  \centering
  \ifuseprecompiled
\includegraphics{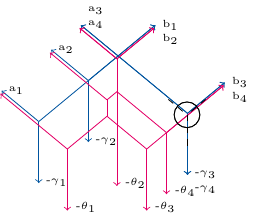}
\else
\setlength\figureheight{3cm}
\setlength\figurewidth{4cm}
\tikzsetnextfilename{plot_non_triv_feasible}
\input{tikz_base_files/plot_non_triv_feasible/plot_non_triv_feasible.tikz}
\fi
\caption{\label{fig:plot_not_triv_feasible} A $(4,2)$-hive consisting of two coupled honeycombs (blue for $\gamma$, magenta for $\theta$),
which are slightly shifted for better visibility, generated by \cref{alg:linprog}. Note that some lines have multiplicity $2$.
The coupled boundary values are given by $a = (4,1.5,0,0)$ and $b = (3.5,3.5,0,0)$. 
It proves the feasibility of the pair $(\gamma,\theta)$, $\widetilde{\gamma}^2 = (7.5,5,0,0)$, $\widetilde{\theta}^2 = (6,3.5,2,1)$ for $m = 2$,
since $\widetilde{\gamma}^2, \widetilde{\theta}^2 \sim_c a \boxplus b$ (the exponent $^2$ has been skipped for better readability).
Only due to the short, vertical line segment in the middle, the hive does not provide diagonal feasibility.} 
\end{figure}
As another example serves $\gamma_+^2 = (10,2,1,0.25,0.25)$ and $\theta_+^2 = (4,3,2.5,2,2)$.\BC{kh1}{kh} According to
\eqref{eq:n2leqn1}, the pair $(\gamma,\theta)$ is not feasible for $m = 2,3$, but may be feasible for $m = 4$. 
The hive in \cref{fig:plot_triv_feasible_single,fig:plot_triv_feasible} (having been constructed with \cref{alg:linprog}) provides that this is indeed the case.
We further know that the pair is diagonally feasible for $m = 5$ (due the constructive \cref{thm:prop1}). 
\begin{figure}[tbhp]
  \centering
  \ifuseprecompiled
\includegraphics{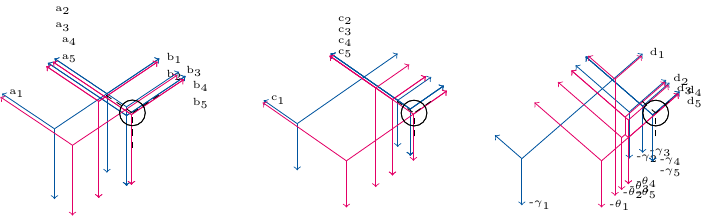}
\else
\setlength\figureheight{2.5cm}
\setlength\figurewidth{12cm}
\tikzsetnextfilename{plot_triv_feasible_n5m4_single}
\input{tikz_base_files/plot_triv_feasible/plot_triv_feasible_n5m4_single.tikz} 
\fi
\caption{\label{fig:plot_triv_feasible_single}  A $(5,4)$-hive consisting of six coupled honeycombs (blue for $\gamma$, magenta for $\theta$),
which are slightly shifted for better visibility, generated by \cref{alg:linprog}. Note that some lines have multiplicity larger than $1$.
Also, in each second pair of honeycombs, the roles of boundaries $\lambda$ and $\mu$ have been switched (which we can do due to the symmetry regarding $\boxplus$), such that the honeycombs can be
combined to a single diagram as in \cref{fig:plot_triv_feasible}. This means that the south rays of an odd numbered pair are always connected
to the north-east (instead of north-west) rays of the consecutive pair.
The boundary values are given by \new{$a = (2,0.25,0.25,0,0)$, $b = (1,1,0.25,0,0)$, $c = (4,0,0,0,0)$ and $d = (3,1,0.75,0,0)$.
It proves the feasibility of the pair $(\gamma,\theta)$, $\widetilde{\gamma}^2 = (10,2,1,0.25,0.25)$, $\widetilde{\theta}^2 = (4,3,2.5,2,2)$ }for $m = 4$,
since both $\widetilde{\gamma}^2, \widetilde{\theta}^2 \sim_c a \boxplus b \boxplus c \boxplus d$ (the exponent $^2$ has been skipped for better readability).}
\end{figure}

\begin{figure}[tbhp]
  \centering
  \ifuseprecompiled
\includegraphics{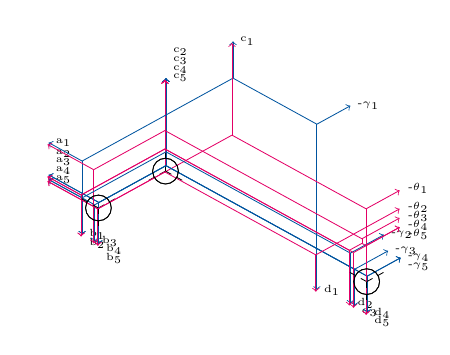}
\else
\setlength\figureheight{6cm}
\setlength\figurewidth{8cm}
\tikzsetnextfilename{plot_triv_feasible_n5m4}
\input{tikz_base_files/plot_triv_feasible/plot_triv_feasible_n5m4.tikz} 
\fi
\caption{\label{fig:plot_triv_feasible} The three overlayed honeycomb pairs in \cref{fig:plot_triv_feasible_single} put together 
with respect to their coupling (the exponent $^2$ for $\gamma,\theta$ has been skipped for better readability).}
\end{figure}

\subsection{Hives are polyhedral cones}\label{sec:cones}

As previously done for honeycombs, we also associate hives with certain vector spaces.

\begin{definition}[Hive sets and edge image]
\new{Let $H$ be an $(n,M)$-hive consisting of honeycombs $h^{(1)},\ldots,h^{(M)}$.
We define}
\[\EDGE(H) = (\EDGE(h^{(1)}),\ldots,\EDGE(h^{(M)})) \in \mathbb{R}^{N \times M} \]
as the collection of constant coordinates of all edges appearing in the honeycombs within the hive $H$.
Although defined via the abstract set $B$ (in \cref{def:hive}), we let
$\sim_S$ act on the related edge coordinates as well.
For $H \in \HIVE_{n,M}(\sim_S)$, we then define the edge image as $\EDGE_S(H) \in \bigslant{\mathbb{R}^{N \times M}}{\sim_S} \cong \mathbb{R}^{N^{\ast}}$, in which coupled boundaries
are assigned the same coordinate. 
\end{definition}

\begin{theorem}[Hive sets are described by polyhedral cones]\label{lem:L1L2}
\hspace*{0cm}
\begin{itemize}
 \item The hive set $\HIVE_{n,M}(\sim_S)$, is a closed, convex, polyhedral cone, i.e. there exist matrices $L_1,L_2$ s.t.
$\EDGE_S(\HIVE_{n,M}(\sim_S)) = \{ x \mid L_1 x \leq 0, \ L_2 x = 0 \}$.
 \item Each fiber of $\delta_P$ (i.e. a set of hives with structure $\sim_S$ and boundary \old{function }$f_P$), forms a closed, convex polyhedron,
 i.e. there exist matrices $L_1,L_2,L_3$ and a vector $b$ s.t.
$\EDGE_S(\delta_P^{-1}(f_P)) = \{ x \mid L_1 x \leq 0, \ L_2 x = 0, \ L_3 x = b \}$.
\end{itemize}
\end{theorem}
\begin{proof}
Each honeycomb of a hive follows its linear constraints. The hive structure and identification of coordinates as one and the same by $\sim_S$
only imposes additional linear constraints. The rest is elementary geometry.
\end{proof}
\begin{corollary}\label{cor:maincor}
The boundary set 
\[\BDRY_{n,M}(\sim_S) := \{ \mathrm{image}(f_P) \in (\mathcal{D}^n)^{P} \mid f_P = \delta_P(H),\ H \in \HIVE_{n,M}(\sim_S)\}\]
forms a closed, convex, polyhedral cone. 
This hence also holds for any intersection with, or projection to a
lower dimensional subspace. 
\end{corollary}
\begin{proof}
The boundary set is given by the projection of $\EDGE_S(\HIVE_{n,M}(\sim_S))$ to the subset of coordinates associated to the ones in $P$.
The proof is finished, since projections to fewer coordinates of closed, convex, polyhedral cones are again such cones.
The same holds for intersections with subspaces.
\end{proof}

\section{\new{Cones of squared feasible values}}\label{sec:coneofsfv}

\new{The following fact has already been established in \cite{DaHa05_Qua}, but also follows from the previous \cref{cor:maincor}. \BC{z4}{z}
\begin{corollary}[Squared feasible pairs form cones]\label{cor:sqfcone}
 Let $m,r_1,r_2 \in \N$. The set of squared feasible pairs $\mathcal{F}^2_{m,(r_1,r_2)}$ (cf. \cref{def:setoffeaspairs}) is a closed, convex, polyhedral cone, embedded into $\R^{r_1 + r_2}$.
 If $r_1 \leq m r_2$ and $r_2 \leq m r_1$, then its dimension is $r_1 + r_2 - 1$.
 Otherwise, $\mathcal{F}^2_{n,(r_1,r_2)} \cap \D^{r_1}_{>0} \times \D^{r_2}_{>0}$ is empty.
\end{corollary}
\begin{proof}
By \cref{cor:maincor} and \cref{thm:relaxed} it directly follows that $\mathcal{F}^2_{m,(r_1,r_2)}$ is a closed, convex, polyhedral cone.
For the first case, it only remains to show that
the cone has dimension $r_1 + r_2 - 1$, or equivalently, it contains as many linearly independent vectors. These are however already
given by the examples carried out in \cref{lem:trivial}. 
From \cref{lem:feasofapair}, it directly follows that if $(\gamma,\theta)$ is feasible for $m$, 
then it must hold $\degree(\gamma) \leq m \degree(\theta)$ and $\degree(\theta) \leq m \degree(\gamma)$, which provides the second case.
%
%
\end{proof}
The implication for the original TT-feasibility then is:
\begin{corollary}[Cone property for higher order tensors]\label{cor:Pyt}\BC{w1}{w}
For $d \in \N$, let both $\sigma, \tau \in (\D^{\infty}_{\geq 0})^{d-1}$ be feasible
for $n \in \N^d$ (in the sense of \cref{def:deffeasibility}). Then $\upsilon$, $(\upsilon^{(\mu)})^2 := (\sigma^{(\mu)})^2 + (\tau^{(\mu)})^2$, $\mu = 1,\ldots,d-1$,
is feasible for $n$ as well.
\end{corollary}
More general, squared feasible TT-singular values form a closed, convex, polyhedral cone. Its H-description
is the collection of linear constraints for the pairs $(\sigma^{(\mu-1)},\sigma^{(\mu)})$.
\begin{proof}
 Due to \cref{cor:pairfeas}, it only remains to show that each pair $(\upsilon^{(\mu-1)},\upsilon^{(\mu)})$ is feasible for $n_\mu$, $\mu = 1,\ldots,d$.
 For each single $\mu$, this follows directly from \cref{cor:sqfcone}.%
\end{proof}
}%
\subsection{\new{Necessary inequalities}}\label{sec:applhives}
%
While for each specific $m$ and $r_1$, the results in \cite{DaHa05_Qua} \BC{z3}{z}%
allow to calculate the $H$-description of the cone $\mathcal{F}^2_{m,(r_1,m r_1)}$ (i.e. a set of necessary and sufficient inequalities),
we will concern ourselves with possibly weaker, but generalized statements for arbitrary $m$ in this section.
In the subsequent \cref{sec:vdescr}, we will derive a $V$-description of $\mathcal{F}^2_{m,(m,m^2)}$ (i.e. a set of generating
vertices).
\begin{lemma}\label{lem:lem35analog}\BC{ki1}{ki}
For $n,m \in \N$, let $T^{(j)}$, $I^{(j)} \subset \{1,\ldots,n\}$ be sets of equal cardinality, $j = 1,\ldots,m$,
with $T^{(1)} = I^{(1)}$ and $\Delta T^{(j)} \sim_c \Delta T^{(j-1)} \boxplus \Delta I^{(j)}$ (cf. \cref{thm:desTrn}) for $j = 2,\ldots,m$.
Then provided $\zeta \sim_c a^{(1)} \boxplus \ldots \boxplus a^{(m)}$, the inequality
\begin{align} \label{eq:necineq}
 \sum_{i \in T^{(m)}} \zeta_i & \leq \sum_{j = 1}^m \sum_{i \in I^{(j)}} a^{(j)}_{i}
\end{align}
holds true, for every $a^{(j)}, \zeta \in \mathcal{D}^n$, $j = 1,\ldots,m$.
If \cref{eq:necineq} holds as equality, then already $\zeta|_{T^{(m)}} \sim_c a^{(1)}|_{I^{(1)}} \boxplus \ldots \boxplus a^{(m)}|_{I^{(m)}}$
and $\zeta|_{(T^{(m)})^c} \sim_c a^{(1)}|_{(I^{(1)})^c} \boxplus \ldots \boxplus a^{(m)}|_{(I^{(m)})^c}$. (cf. \cref{lem:reduc}).
\end{lemma}
\begin{proof}
 The statement \cref{eq:necineq} follows inductively, if for each $j = 2,\ldots,m$,
\begin{align}\label{eq:induc1}
\sum_{i \in T^{(j)}} \nu_{i} & \leq \sum_{i \in T^{(j-1)}} \lambda_{i} + \sum_{i \in I^{(j)}} \mu_{i}
\end{align}
is true whenever $\nu \sim_c \lambda \boxplus \mu$. By \cref{thm:desTrn}, this holds since by assumption
$\Delta T^{(j)} \sim_c \Delta T^{(j-1)} \boxplus \Delta I^{(j)}$ for $j = 2,\ldots,m$. If \cref{eq:necineq} holds as equality,
then all single inequalities \cref{eq:induc1} must hold as equality, and hence \cref{lem:reduc} can be applied inductively as well.
\end{proof}
\begin{theorem}\label{cor:feasineq}
 In the situation of \cref{lem:lem35analog}, let $\widehat{T}$ and $\widehat{I}$ fulfill the same assumptions as $T$ and $I$.
 Let further $I^{(j)} \cap \widehat{I}^{(j)} = \emptyset$, $j = 1,\ldots,m$.
 If the pair $(\gamma,\theta) \in \mathcal{D}^{\infty}_{\geq 0} \times \mathcal{D}^{\infty}_{\geq 0}$ is feasible for $m$, then 
 \begin{align}\label{eq:feasineq}
  \sum_{i \in T^{(m)}} \gamma^2_{i} \leq \sum_{i \in \{1,\ldots,\degree(\theta)\} \setminus \widehat{T}^{(m)}} \theta^2_{i}
 \end{align}
 must hold true. If \cref{eq:feasineq} holds as equality, then $((\gamma|_{T^{(m)}},0,\ldots), \ (\theta|_{(\widehat{T}^{(m)})^c},0,\ldots))$
 and $((\gamma|_{(T^{(m)})^c},0,\ldots), \ (\theta|_{\widehat{T}^{(m)}},0,\ldots))$ are already feasible.
\end{theorem}
Together with \cref{eq:overlay} this also implies that the corresponding hive
is an overlay of two smaller hives modulo zero boundaries.
\begin{proof}
Let $n \geq \max(T^{(m)}),\degree(\gamma),\degree(\theta)$.
As $(\gamma,\theta)$ is feasible, due to \cref{lem:lem35analog}, the inequality \cref{eq:necineq} holds for some joint eigenvalues $a^{(1)},\ldots,a^{(m)} \in \mathcal{D}^n_{\geq 0}$
for both $\zeta := \widetilde{\gamma}^2 = (\gamma_1^2,\ldots,\gamma_n^2)$, $T$, $I$ and $\widehat{\zeta} = \widetilde{\theta}^2 := (\theta_1^2,\ldots,\theta_n^2)$, $\widehat{T}$, $\widehat{I}$.
Furthermore, we have 
$\sum_{i = 1}^{n} \theta_i^2 = \sum_{i = 1}^{n} a^{(1)}_i + \ldots + \sum_{i = 1}^{n} a^{(m)}_i$.
Subtracting \cref{eq:necineq} for $\widehat{\zeta}$ from this equality yields
\begin{align}
\sum_{i \notin \widehat{T}^{(m)}} \theta^2_i \overset{n \geq \degree(\theta)}{=} \sum_{i \in \{1,\ldots,n\} \setminus \widehat{T}^{(m)}} \theta^2_i & \overset{\cref{eq:necineq} \mbox{ for } \widehat{\zeta}}{\geq}
\sum_{j = 1}^m \sum_{i \in \{1,\ldots,n\} \setminus \widehat{I}^{(j)}} a^{(j)}_{i} \\
& \overset{a_i^{(j)} \geq 0}{\geq} \sum_{j = 1}^m \sum_{i \in I^{(j)}} a^{(j)}_{i} 
\overset{\cref{eq:necineq} \mbox{ for } \zeta}{\geq} \sum_{i \in T^{(m)}} \gamma^2_{i}.
\label{eq:lasteqinsetineq1}
\end{align}
This finishes the first part.
In case of an equality, since the second ``$\geq$'' must hold as equality, we have $a^{(j)}|_{\{1,\ldots,n\} \setminus \widehat{I}^{(j)}} = (a^{(j)}|_{I^{(j)}},0,\ldots)$ and
$a^{(j)}|_{\{1,\ldots,n\} \setminus I^{(j)}} = (a^{(j)}|_{\widehat{I}^{(j)}},0,\ldots)$ for each $j = 1,\ldots,m$.
Furthermore, the first and third ``$\geq$'' in \cref{eq:lasteqinsetineq1} must hold as equality as well. Hence,
the latter statement in \cref{lem:lem35analog} can be applied to the inequalities \cref{eq:necineq} for both $\zeta$ and $\widehat{\zeta}$, such that we can conclude the latter statement in this corollary.
\end{proof}
\begin{corollary}[A set of inequalities for feasible pairs]\label{prop:setofineq}
Let $p^{(1)} \ \dot{\cup} \ p^{(2)} = \mathbb{N}$ be two disjoint sets, with $p^{(1)}$ finite of size $r$.
If $(\gamma,\theta) \in \mathcal{D}^{\infty}_{\geq 0} \times \mathcal{D}^{\infty}_{\geq 0}$ is feasible for $m \in \mathbb{N}$, then it holds
($p^{(u)}_i$ being the $i$-th smallest element)
\begin{align*}
 \sum_{i \in P_m^{(1)}} \gamma^2_{i} \leq \sum_{i \notin P_m^{(2)}} \theta^2_i, \quad P_m^{(u)} := \{ m(p^{(u)}_i - i) + i \mid i = 1,2,\ldots \}, \ u = 1,2.
\end{align*}
\end{corollary}
\begin{proof}
 Let $n \geq \max(P_m^{(1)}),\degree(\gamma),\degree(\theta)$. Let further $\widetilde{P}_j^{(2)}$ contain the
$\widehat{k}$ smallest elements of $P_j^{(2)}$, where $\widehat{k}$ is the number of elements in $P_m^{(2)} \cap \{1,\ldots,n\}$, and let $\widetilde{P}_j^{(1)} = P_j^{(1)}$, $j = 1,\ldots,m$.
Thereby $\widetilde{P}_1^{(1)} = p^{(1)} = P_1^{(1)}$ and $\widetilde{P}_1^{(2)} \subset p^{(2)} = P_1^{(2)}$.
We have the following (diagonal) matrix identities
\begin{align*}
 \diag(\widetilde{P}_j^{(u)}) - \diag(1,\ldots,\ell) & = \diag(\widetilde{P}_{j-1}^{(u)}) + \diag(\widetilde{P}_{1}^{(u)}) - 2\diag(1,\ldots,\ell) \\
 \Leftrightarrow j(p^{(u)}_i - i) & = (j-1)(p^{(u)}_i - i) + (p^{(u)}_i - i), \quad i = 1,\ldots,\ell,\ \ell = |\widetilde{P}_j^{(u)}|
\end{align*}
where the diagonal elements are placed in ascending order. Hence, $\bigtriangleup \widetilde{P}_j^{(u)} \sim_c \bigtriangleup \widetilde{P}_{j-1}^{(u)} \boxplus \bigtriangleup \widetilde{P}_1^{(u)}$
for $j = 2,\ldots,m$, $u \in \{1,2\}$.
For $T^{(j)} := \widetilde{P}_j^{(1)}$, $I^{(j)} := \widetilde{P}_1^{(1)}$ and $\widehat{T}^{(j)} := \widetilde{P}_j^{(2)}$, $\widehat{I}^{(j)} := \widetilde{P}_1^{(2)}$,
we can apply \cref{cor:feasineq} to obtain the desired statement.
\end{proof}
Among the various inequalities contained in \cref{prop:setofineq},\BC{f1_}{f_}
the following two correspond to early mentioned inequalities for Weyl's problem.
The first case is \cref{eq:n2leqn1} and is also referred to as the \textit{basic inequalities} in \cite{DaHa05_Qua}.

\begin{corollary}[Ky Fan analogue for feas. pairs]\label{cor:kyfan}
The choice $a^{(1)} = \{1,\ldots,r\}$ in \cref{prop:setofineq} yields the inequality
 $\sum_{i = 1}^r \gamma_i^2 \leq \sum_{i = 1}^{mr} \theta_i^2.$
\end{corollary}

\begin{corollary}[Weyl analogue for feas. pairs]\label{cor:prop2}
The choice $a^{(1)} = \{r+1\}$ in \cref{prop:setofineq} yields the inequality
$ \gamma_{rm+1}^2 \leq \sum_{i = r+1}^{r+m} \theta_i^2.$
\end{corollary}

The QMP article \cite{DaHa05_Qua} explicitly provides the derivation for the case $\degree(\gamma) \leq 3$ and $m = 2$. \BC{z1}{z}
Thereby, the necessary (and sufficient) inequalities for the feasibility of $(\gamma,\theta)$, apart from the trace property, are as follows:
\cref{cor:kyfan} for $r = 1,2$; \cref{cor:prop2} for $r = 1$ and $\gamma^2_2 + \gamma^2_3 \leq \theta^2_1 + \theta^2_2 + \theta^2_3 + \theta^2_6$.
The last inequality is not included in \cref{prop:setofineq}, but can be derived from \cref{cor:feasineq} and be generalized in different ways. For example,
for $I^{(1)} = I^{(2)} = \{1,3\}$, $T^{(2)} = \{2,3\}$, $\widehat{I}^{(1)} = \widehat{I}^{(2)} = \{2,4,5,6,\ldots\}$, $\widehat{T}^{(2)} = \{4,5,7,8,\ldots\}$ and $I^{(j)} = \{1,2\}$, $T^{(j)} = \{2,3\}$, $\widehat{I}^{(j)} = \{3,4,5,6,\ldots\}$, $\widehat{T}^{(j)} = \{2j,2j+1,2j+3,2j+4,\ldots\}$, $j = 3,\ldots,m$,
(where we add the same amount of arbitrarily many consecutive numbers in $\widehat{I}^{(j)}$ and $\widehat{T}^{(j)}$)
one can conclude that
 $\gamma^2_2 + \gamma^2_3 \leq \sum_{i = 1}^{2m-1} \theta^2_i + \theta^2_{2m+2}$ whenever $(\gamma,\theta)$ is feasible for $m$.
\cref{cor:feasineq} does however not provide when this generalized inequality is redundant to other necessary ones.\\\\
The right sum in \cref{prop:setofineq} has always $m$-times as many summands as the left sum. 
For these inequalities, it further holds $\sum_{i \notin P_m^{(2)}} i - \sum_{i \in P_m^{(1)}} i = \sum_{i = k+1}^{mk} i = \frac{k(m-1)((m+1)k +1)}{2}$, 
where $k = |P_m^{(1)}|$. We can however only conjecture that this holds in general for every inequality in the
$H$-description of $\mathcal{F}^2_{m,(r_1,mr_1)}$.

\subsection{\new{Vertex description of \texorpdfstring{$\mathcal{F}^2_{m,(m,m^2)}$}{Fmmm2}}}\label{sec:vdescr}\BC{ka2}{ka}%
\def\ones#1#2{{#1}_{\# #2}}%
We revisit the special case \cref{eq:n2leqn1} and derive the vertex description of the corresponding cone $\mathcal{F}^2_{m,(m,m^2)}$ (cf. \cref{def:setoffeaspairs}).
In this section, for $a,b \in \N \cup \{0\}$, let therefor $(\ones{a}{b}) = (a,\ldots,a) \in \mathcal{D}^b_{\geq 0}$ (length $b$).
\begin{lemma}\label{lem:abfeas}
 Let $\alpha, \beta, m \in \N$, $\beta \leq m$, $\alpha \leq \beta m$, $\gamma_+^2 = (\ones{\alpha}{\beta})$ and $\theta_+^2 = (\ones{\beta}{\alpha})$.
 Then $(\gamma,\theta)$ is feasible for $m$. 
\end{lemma}
\begin{proof}
We prove by induction over $m$.
Without loss of generality, we may assume $\alpha > \beta$ by which $\alpha = k \beta + t$ for unique natural numbers $k < m, t < \beta$.
Considering \cref{rem:diag}, it suffices to show that for
$\widetilde{\gamma}^2 := \gamma_+^2 - (t,\ones{\beta}{\beta-1}) = (k\beta,\ones{(\alpha-\beta)}{\beta-1})$ and $\widetilde{\theta}^2 := \theta_+^2 - (0,\ldots,0,t,\ones{\beta}{\beta-1}) = (\ones{\beta}{\alpha-\beta},\beta-t,\ones{0}{\beta-1})$ 
the pair $((\widetilde{\gamma},0,\ldots), (\widetilde{\theta},0,\ldots))$ is feasible for $m-1$.
In order to show this, we split $\widetilde{\gamma} = (\widetilde{\gamma}_{(1)},\widetilde{\gamma}_{(2)})$, $\widetilde{\theta} = (\widetilde{\theta}_{(1)},\widetilde{\theta}_{(2)})$ into two pairs
$\widetilde{\gamma}^2_{(1)} := (k\beta)$, $\widetilde{\gamma}^2_{(2)} := (\ones{\beta}{k})$ and $\widetilde{\theta}^2_{(1)} := (\ones{(\alpha-\beta)}{\beta-1},0,\ldots,0)$, $\widetilde{\theta}^2_{(2)} := (\ones{\beta}{v},\beta-t,0,\ldots,0)$ with $v = \alpha-\beta-k = (k-1)(\beta-1) + (t-1)$.
We can then, considering overlays of honeycombs, treat both pairs independently.
While $((\widetilde{\gamma}_{(1)},0,\ldots), (\widetilde{\theta}_{(1)},0,\ldots))$ is feasible for $k \leq m-1$, in the second case,
$(\widetilde{\gamma}^2_{(2)},\widetilde{\theta}^2_{(2)})$ is a convex combination of $(\ones{(v+1)}{\beta-1})$, $(\ones{(\beta-1)}{v+1})$
and $(\ones{v}{\beta-1})$, $(\ones{(\beta-1)}{v})$. Since $\beta-1 \leq m-1$ and $v \leq v+1 \leq (m-1)(\beta-1)$, the proof is finished by induction.
\end{proof}
%
The following theorem has priorly been conjectured by \cite{DaHa05_Qua} and proven by \cite{LiPoWa14_Ran}. \BC{z2}{z}%
We prove it in a way which allows to identify all vertices as in \cref{completevert}.
\begin{theorem}\label{thm:restatement}
 Let $(\gamma,\theta) \in \mathcal{D}^{\infty}_{\geq 0} \times \mathcal{D}^{\infty}_{\geq 0}$ and $m \in \N$. 
 If $\degree(\gamma) \leq m$ and if all Ky Fan inequalities (\cref{cor:kyfan}) as well as the trace property $\|\gamma\|_2 = \|\theta\|_2$ hold, then the pair is feasible for $m$.
\end{theorem}
\begin{proof}
Here, we denote the Ky Fan inequality (\cref{cor:kyfan}) for $r$ with $K_r$, and in case of an equality we say $E_r$ holds.
 Due to $K_m$, $\degree(\gamma) \leq m$ and the trace property, $E_m$ and $\degree(\theta) \leq m \degree(\gamma)$ must be true.
 For fixed $m$, we prove by induction over $\degree(\gamma)+\degree(\theta)$. Let $0 \leq k < m$ be the largest number for which $E_k$ is fulfilled
 and let $\alpha = \degree(\theta) - mk$ as well as $\beta = \degree(\gamma) - k$. We define
 $(\widehat{\gamma}^2 \mid \widehat{\theta}^2) := (\gamma_+^2 \mid \theta_+^2) - f \cdot (\ones{m \beta}{k},\ \ones{\alpha}{\beta} \mid \ones{\beta}{m k},\ \ones{\beta}{\alpha})$, $f > 0$.
 Then $E_k$ and $K_j$, $j < k$, are true for $(\widehat{\gamma} \mid \widehat{\theta})$ for all $f > 0$.
 Further, as long as $K_{k+1}$ holds for $(\widehat{\gamma} \mid \widehat{\theta})$ (which it does for any $f > 0$ if $k = m-1$), then due to $K_{k-1}$ and $E_k$ it follows that $\widehat{\gamma}_{k+1} \leq \widehat{\gamma}_{k}$.
 Hence, $f$ can be chosen such that $K_i$, $i = 1,\ldots,m-1,$ and $(\widehat{\gamma} \mid \widehat{\theta}) \in \mathcal{D}_{\geq 0}^{\beta} \times \mathcal{D}_{\geq 0}^{\alpha}$ as well as either $(i)$ $E_{j}$ for at least one $j$, $k < j < m$, or $(ii)$ $\widehat{\gamma}_\beta = 0 \vee \widehat{\theta}_\alpha = 0$.
 In case of $(i)$, we can repeat the above construction for increased $k$ until $k = m-1$ and hence $(ii)$ remains the sole option. In that case, we are finished by induction.
\end{proof}
%
%
%
\begin{corollary}\label{completevert}
 A complete vertex description of $\mathcal{F}_{m,(m,m^2)}$ is given by 
 \begin{align*}
\mathcal{V} = \{ & (\widetilde{\gamma}, \widetilde{\theta}) \in \mathcal{F}_{m,(m,m^2)} \mid \widetilde{\gamma}_+^2 = (\ones{m \beta}{k},\ \ones{\alpha}{\beta}), \ \widetilde{\theta}_+^2 = (\ones{\beta}{m k},\ \ones{\beta}{\alpha}), \\ 
 & \quad k \in \{0,\ldots,m - \beta\},\ \alpha, \beta \in \N,\ \beta \leq m, \ \alpha \leq \beta m; \ \mbox{and } k = 0 \mbox{ if } \alpha = \beta m \}
 \end{align*}
\end{corollary}
 A short calculation shows that the number of vertices $|\mathcal{V}|$ is given by a polynomial with leading monomial $m^4/6$.
\begin{proof}
The proof of \cref{thm:restatement} is constructive and decomposes a squared feasible pair into a convex
combination of squared feasible pairs in $\mathcal{V}$. It hence remains to show that the elements of $\mathcal{V}$ are vertices. 
Given any two elements $v = v(k_1,\alpha_1,\beta_1)$, $w = w(k_2,\alpha_2,\beta_2)$, $v^2,w^2 \in \mathcal{V}$, let $y^2_f = v^2 - f \cdot w^2$, $f > 0$.
For $y_f \in \mathcal{D}^m_{\geq 0} \times \mathcal{D}^{m^2}_{\geq 0}$ to be true, we must have $m k_1 + \alpha_1 = m k_2 + \alpha_2$ as well as either $(i)$ $k_1 = k_2$
and $\beta_1 = \beta_2$ or $(ii)$ $k_2 = 0$ and $\beta_1 + k_1 = \beta_2$. In the second case, $y_f$ would violate $K_{k_1}$ if $k_1 \neq 0$.
If $y^2_f$ is again a convex combination of elements in $\mathcal{V}$, $y_f$ must be feasible. Due to the above, it then however follows that $v = w$, $y^2 = (1-f) v^2$.
In other words, $v^2$ can not be a convex combination of other elements in $\mathcal{V}$.
%
%
%
\end{proof}
For example, all $7$ vertices $v^2_1,\ldots,v^2_7$ of $\mathcal{F}^2_{2,(2,4)} \ (m = 2)$ are given through 
\begin{align*}
\left[\begin{smallmatrix}
 1 & 1 \\
 0 & 0 \\
 & 0 \\
 & 0 
\end{smallmatrix}\right],
\left[\begin{smallmatrix}
 2 & 1 \\
 0 & 1 \\
 & 0 \\
 & 0 
\end{smallmatrix}\right],
\left[\begin{smallmatrix}
 1 & 2 \\
 1 & 0 \\
 & 0 \\
 & 0 
\end{smallmatrix}\right],
\left[\begin{smallmatrix}
 1 & 1 \\
 1 & 1 \\
 & 0 \\
 & 0 
\end{smallmatrix}\right],
\left[\begin{smallmatrix}
 2 & 1 \\
 1 & 1 \\
 & 1 \\
 & 0 
\end{smallmatrix}\right],
\left[\begin{smallmatrix}
 3 & 2 \\
 3 & 2 \\
 & 2 \\
 & 0 
\end{smallmatrix}\right] \mbox{and}
\left[\begin{smallmatrix}
 2 & 1 \\
 2 & 1 \\
 & 1\\
 & 1 
\end{smallmatrix}\right].
\end{align*}
For $m = 3$, we already have $27$ vertices. Although all these vertices happen to be diagonally feasible, this
is not the case in general. For example, $(\ones{5}{3} \mid \ones{3}{5},\ones{0}{4}) \in \mathcal{F}^2_{3,(3,9)}$ 
is a vertex, but it is easy to show that it is not diagonally feasible.
For $(\gamma,\theta)$ as in \cref{eq:not_triv_pair}, $\gamma_+^2 = (7.5,5)$, $\theta_+^2 = (6,3.5,2,1)$, we have $(\gamma_+,\theta_+) = 1.5 v^3_1 + 0.5 v^3_2 + 1.5 v^2_4 + v^2_5 + v^2_7$.\BC{kj1}{kj}

\section{\new{TFP }algorithms}\label{sec:algo}
Matlab implementations of algorithms mentioned in this work can be found under the name {\tt TT-feasibility-toolbox}
or directly at \begin{center} \url{https://git.rwth-aachen.de/sebastian.kraemer1/TT-feasibility-toolbox}. \end{center}
The description in \cref{lem:L1L2} yields the straightforward \cref{alg:linprog} to determine the minimal
value $\new{m}$ for which some pair $(\gamma,\theta) \in \D^{\infty}_{\geq 0} \times \D^{\infty}_{\geq 0}$
is feasible. The summed up length of all (inner) edges is minimized, since then the algorithm tends
to return a hive from which diagonal feasibility can be read off \new{(cf. \cref{lem:trivial})}.\old{(that is, eigenvalues and permutations, from
which a core can be constructed)}
%
%
 \begin{algorithm}
  \caption{Linear programming check for feasibility \label{alg:linprog}}
  \begin{algorithmic}[1] 
  \REQUIRE{$(\widetilde{\gamma},\widetilde{\theta}) \in \D^{r}_{\geq 0} \times \D^{r}_{\geq 0}$ with $\|\widetilde{\gamma}\|_2 = \|\widetilde{\theta}\|_2$ for some $r \in \mathbb{N}$}
  \FOR{$\new{m} = 2\ldots$}
  \STATE{as in \cref{lem:L1L2}, set $L$ such that $\EDGE_S(\delta_P^{-1}(f_P)) = \{ x \mid L_1 x \leq 0, \ L_2 x = 0, \ L_3 x = b \}$ for 
  the hive $H$ as in \cref{thm:relaxed}}
  \STATE{use a linear programming algorithm to minimize $Fx$ subject to $x \in \EDGE_S(\delta_P^{-1}(f_P))$,
  where $F$ is the vector for which $Fx \in \R_{\geq 0}$ is the summed up length of all (inner) edges in $H$}
  \IF{no solution exists}
  \STATE{continue with $\new{m}+1$}
  \ELSE
    \RETURN{minimal number $\new{m} \in \mathbb{N}$ for which $(\gamma,\theta)$ is feasible and a corresponding
          $(r,2\new{(m-1)})$-hive $H$ with minimal total edge length}
  \ENDIF
  \ENDFOR
  \end{algorithmic}
\end{algorithm}
%
\mbox{\Cref{alg:linprog}} always terminates for at most $\new{m} = \max(\degree(\gamma),\degree(\theta))$ due to \cref{lem:trivial}.
In practice, a slightly different coupling of boundaries is used (cf. \cref{fig:plot_triv_feasible_single}), since then the entire hive can be 
visualized in $\R^2$. For that, it is required to rotate and mirror some of the honeycombs (cf. \cref{fig:plot_triv_feasible}).
Depending on the linear programming algorithm, the input may be too badly conditioned\BC{i1_}{i_} to
allow a verification with satisfying residual.
 \begin{algorithm}
  \caption{Heuristic check for numerical feasibility \label{alg:numfea}}
  \begin{algorithmic}[1]
  \REQUIRE{$(\gamma_+,\theta_+) \in \D^{r_1}_{> 0} \times \D^{r_2}_{> 0}$ for some $r_1,r_2 \in \mathbb{N}$
  and a natural number $\new{m}$ \\ (as well as $\mathrm{tol} > 0$, $\mathrm{iter_{max}} > 0$)}
  \STATE{initialize a core $H^{(1)}_1$ of length $\new{m}$ and size $(r_1,r_2)$ randomly}
  \STATE{set $\gamma^{(0)}, \theta^{(0)} \equiv 0$, $\mathrm{relres} = 1$ and $k = 0$}
  \WHILE{$\mathrm{relres} > \mathrm{tol}$ and $k \leq \mathrm{iter_{max}}$}
    \STATE{$k = k + 1$}
  \STATE{calculate the SVD and set $U_1 \Theta^{(k)} V_1^T = \lhb{H^{(k-1)}_1}$ }
  \STATE{set $H^{(k)}_2$ via $\lhb{H^{(k)}_2} = U_1 \Theta$}
  \STATE{calculate the SVD and set $U_2 \Gamma^{(k)} V_2^T = \rhb{H^{(k)}_2}$}
  \STATE{set $H^{(k)}_1$ via $\rhb{H^{(k)}_1} = \Gamma V_2^T$}
  \STATE{$\mathrm{relres} = \max\left(\max_{i=1,\ldots,r_1}(|\gamma^{(k)}_i/\gamma_i-1|),\max_{i=1,\ldots,r_2}(|\theta^{(k)}_i/\theta_i-1|)\right)$}
  \ENDWHILE
  \IF{$\mathrm{relerr} \leq \mathrm{tol}$}
    \RETURN{$H^{\ast} = \Gamma^{-1} \ H_1^{(k)} \ \Theta^{-1}$}
    \STATE{$(\gamma,\theta)$ is (numerically) feasible for $\new{m}$}
   \ELSE
    \STATE{$(\gamma,\theta)$ is \textit{likely} to not be feasible for $\new{m}$}
  \ENDIF
  \end{algorithmic}
\end{algorithm}
The simple and heuristic \cref{alg:numfea} can 
be more reliable. As we have seen, we can restrict ourselves to $\K = \R$ (cf. \cref{thm:fultonsym}).
Fixpoints of the iteration are \old{given by the cores }\new{cores $H \in (\R^{r_1 \times r_2})^{\{1,\ldots,m\}}$ }for which
$H \Theta^{-1}$ is left-orthonormal and $\Gamma^{-1} H$ is right-orthonormal.
Hence $H^{\ast} = \Gamma^{-1} \ H \ \Theta^{-1}$ is a core for which
 $\Gamma H^{\ast}$ is left-orthonormal and $H^{\ast} \Theta$ is right-orthonormal, as
 required by \cref{lem:feasofapair}. Furthermore, the iterates cannot
 diverge in the following sense:
 
 \begin{lemma}[Behavior of \cref{alg:numfea}]
 For every $k > 1$ it holds $\|\gamma^{(k)}-\gamma_+\|_2 \leq \|\theta^{(k)}-\theta_+\|_2$ as well as
 $\|\theta^{(k)}-\theta_+\|_2 \leq \|\gamma^{(k-1)}-\gamma_+\|_2$.
 \end{lemma}
 \begin{proof}
 We only \old{treat }\new{consider }the first case, since the other one is analogous. Let $k > 1$ be arbitrary, but fixed.
 Then \new{in line $5$ of \cref{alg:numfea} we have}
 \[\|\underbrace{U_1 \Theta^{(k)} V_1^T}_{A:=} - \underbrace{U_1 \Theta V_1^T}_{B:=}\|_F = \|\Theta^{(k)} - \Theta\|_F.\]
 $\rhb{A}$ has singular values $\gamma_+$, inherited from the last iteration and
 $\rhb{B}$ has the same singular values as $\rhb{B} \diag(V_1,\ldots,V_1) = \rhb{B V_1} = \rhb{H_2^{(k)}}$, which are given by $\gamma^{(k)}$. 
 It follows by Mirsky's inequality about singular values \cite{Mi1960_Sym} that $\|\gamma^{(k)}-\gamma_+\|_2 \leq \|A - B\|_F = \|\theta^{(k)}-\theta_+\|_2$.
 \end{proof}

 Convergence is hence not assured, but likely in the sense that the perturbation of matrices usually 
 leads to a fractional amount of perturbation of its singular values. To construct an entire tensor,
 the algorithm may be run in parallel for each single core.


\section{Conclusions and outlook}\label{sec:concl}\BC{kk1}{kk}
The simple equivalence between the tensor feasibility (TFP) and quantum marginal problem (QMP) allows for
an interesting interaction between the different perspectives on either side. Through the standard representation, 
the tensor train (TT-)feasibility problem can be decoupled into pairwise problems, by which, firstly, results from the QMP can be applied.
Thereby, the full H-description of the cone of squared TT-feasible values can be calculated in any specific instance.
At the same time, through our alternative consideration of orthogonality constraints on cores,
one can derive universal classes of necessary inequalities for the feasibility of pairs, whereas the concept of hives
yields a corresponding linear programming algorithm. Further on the practical side,
we have introduced simple ways to construct tensors with prescribed singular values in parallel,
based only on the sufficient construction of feasible pairs. 
Given that the concept of a standard representation
is transferable to any hierarchical format, implications for both the TFP and QMP are subject to future research.

\bibliographystyle{siamplain}
\footnotesize
\bibliography{singular_value_connections}
\normalsize


%
%
\end{document}
